\documentclass{article}
\title{The fundamental group as a topological group}
\author{Jeremy Brazas}
\usepackage{stmaryrd}
\usepackage{graphicx}
\usepackage{amssymb,mathrsfs,amsmath,amscd,fullpage,amsthm}
\usepackage{float}
\usepackage[all,cmtip]{xy}
\DeclareMathAlphabet{\mathpzc}{OT1}{pzc}{m}{it}
\usepackage{amsfonts,txfonts,pxfonts,latexsym,wasysym}
\newcommand{\sus}{\Sigma(X_{+})}
\newcommand{\ra}{\rightarrow}

\newcommand{\piz}{\pi_{0}(X)}

\newcommand{\pitsus}{\pi_{1}^{\tau}(\sus)}
\newcommand{\pitx}{\pi_{1}^{\tau}(X)}
\newcommand{\pitxxo}{\pi_{1}^{\tau}(X,x_0)}
\newcommand{\pity}{\pi_{1}^{\tau}(Y)}

\newcommand{\piztop}{\pi_{0}^{qtop}(X)}

\newcommand{\pitopxxo}{\pi_{1}^{qtop}(X,x_0)}
\newcommand{\pitopx}{\pi_{1}^{qtop}(X)}
\newcommand{\pitopy}{\pi_{1}^{qtop}(Y)}

\newcommand{\ox}{\Omega(X)}

\newcommand{\oxxo}{\Omega(X,x_0)}

\newcommand{\px}{P(X)}

\newcommand{\pitop}{\pi_{1}^{qtop}(\sus)}

\newcommand{\rat}{\mathbb{Q}}

\newcommand{\fmx}{F_{M}(X)}
\newcommand{\fmy}{F_{M}(Y)}
\newcommand{\qtg}{\mathbf{qTopGrp}}

\newcommand{\tg}{\mathbf{TopGrp}}

\newcommand{\gwt}{\mathbf{GrpwTop}}
\newcommand{\spaces}{\mathbf{Top}}
\newcommand{\grp}{\mathbf{Grp}}

\newcommand{\bspaces}{\mathbf{Top_{\ast}}}
\newcommand{\hbspaces}{\mathbf{hTop_{\ast}}}

\newcommand{\pit}{\pi_{1}^{\tau}}
\newtheorem{theorem}{Theorem}[section]
\newtheorem{lemma}[theorem]{Lemma}

\newtheorem{proposition}[theorem]{Proposition}
\newtheorem{corollary}[theorem]{Corollary}
\newtheorem{definition}[theorem]{Definition}

\newtheorem{example}[theorem]{Example}
\newtheorem{question}[theorem]{Question}

\newtheorem{remark}[theorem]{Remark}

\newtheorem{statement}[theorem]{Statement}
\newtheorem{vankampenthm}[theorem]{The van Kampen Theorem}
\newtheorem{construction}[theorem]{Construction of $\pit$}

\newtheorem{Topologicalshapegroups}[theorem]{The shape group as a topological group}

\newtheorem{approximation}[theorem]{Approximation of $\tau(G)$}

\begin{document}
\maketitle
\begin{abstract}
This paper is devoted to the study of a natural group topology on the fundamental group which remembers local properties of spaces forgotten by covering space theory and weak homotopy type. It is known that viewing the fundamental group as the quotient of the loop space often fails to result in a topological group; we use free topological groups to construct a topology which promotes the fundamental group of any space to topological group structure. The resulting invariant, denoted $\pi_{1}^{\tau}$, takes values in the category of topological groups, can distinguish spaces with isomorphic fundamental groups, and agrees with the quotient fundamental group precisely when the quotient topology yields a topological group. Most importantly, this choice of topology allows us to naturally realize free topological groups and pushouts of topological groups as fundamental groups via topological analogues of classical results in algebraic topology. 
\end{abstract}
\section{Introduction}
Classical results in basic algebraic topology give that groups are naturally realized as fundamental groups of spaces. For instance, free groups arise as fundamental groups of wedges of circles, any group can be realized as the fundamental group of some CW-complex, and pushouts of groups arise via the van Kampen theorem. This ability to construct geometric interpretations of discrete groups has important applications in both topology and algebra. It is natural to ask if this symbiotic relationship can be extended, via some topological version of the invariant, so that one can study topological groups by studying spaces with homotopy type other than that of a CW-complex and vice versa.\\
\indent This paper is devoted to one such topologically enriched version of the fundamental group. In particular, the fundamental group is endowed with a group topology which can be used to study spaces with complicated local structure. To approach such spaces by way of directly transferring topological data to a homotopy invariant is in contrast with the more historical shape theoretic approach where spaces are approximated by polyhedra and pro-groups take the place of groups. Our view is quickly justified, however, as we find natural connections between locally complicated spaces and the general theory of topological groups not possible in shape theory. For instance, the use of (non-discrete) topological groups in the place of groups allow us to give ``realization" results (See Section 4) in the form of topological analogues of the classical results mentioned above. Additionally, we set the foundation for a theory of generalized covering maps (called semicoverings) to appear in \cite{Brsemi}.\\
\indent In making a choice of topology on $\pi_1$, it is inevitable that we ignore many interesting topologies likely to have their own benefits and uses. The topology introduced in the present paper is partially motivated by the quotient topology used first by Biss \cite{Bi02}. Specifically, $\pi_{1}(X,x_0)$ may be viewed as the quotient space of the loop space $\oxxo$ (with the compact-open topology) with respect to the natural function $\Omega(X,x_0)\ra \pi_{1}(X,x_0)$ identifying path components. This construction results in a functor $\pi_{1}^{qtop}:\mathbf{hTop_{\ast}}\ra \qtg$ from the homotopy category of based spaces to the category of quasitopological groups\footnote{A quasitopological group is a group with topology such that inversion is continuous and multiplication is continuous in each variable. See \cite{AT08} for basic theory.} and continuous homomorphisms \cite{Bi02,CM}, however, recent results indicate that very often $\pi_{1}^{qtop}(X,x_0)$ fails to be a topological group. In fact, this failure occurs even for one-dimensional planar, Peano continua \cite{Fab10} and locally simply-connected (but non-locally path connected) planar sets \cite{Br10.1,Fab11}. For this reason, we refer to $\pi_{1}^{qtop}$ as the \textit{quasitopological fundamental group}.\\
\indent When considering which properties a ``topological" fundamental group should have, the continuity of the natural function $\oxxo\ra \pi_{1}(X,x_0)$ is certainly among the most useful. Indeed, a topology on $\pi_{1}(X,x_0)$ should remember something about the topology of loops representing homotopy classes. By definition, the quotient topology of $\pitopxxo$ is the finest topology on $\pi_{1}(X,x_0)$ with this property. On the other hand, $\pitopxxo$ is not always a topological group. One might then ask if there is a finest \textit{group topology} on $\pi_{1}(X,x_0)$ such that $\oxxo\ra \pi_{1}(X,x_0)$ is continuous. We find that the existence of this topology follows directly from the existence of the free topological groups in the sense of Markov. The resulting topological group, denoted $\pitxxo$, is invariant under homotopy equivalence and retains information beyond the covering space theory of $X$. The construction and basic theory of $\pitxxo$ appears in Section 3 after notation and preliminary facts are discussed in Section 2.\\
\indent The appearance of free topological groups in \cite{Br10.1} also motivates the construction of $\pitxxo$. There is a vast literature of free topological groups and free topological products, however, explicit descriptions of these groups are often quite complicated. While it is unlikely such groups arise naturally in shape theory, we find they appear with great generality in topological analogues of classical computational results involving $\pi_{1}^{\tau}$. For instance, just as the fundamental group of a wedge of circles $\bigvee_{X}S^1$ is free on a discrete set $X$, $\pi_{1}^{\tau}$ evaluated on a ``generalized wedge of circles" $\sus$ on arbitrary $X$ is free topological on the path component space of $X$ (Theorem \ref{theoremfreetopologicalgroups}). Secondly, mimicking the usual proof that every group is realized as a fundamental group, it is possible to realize every topological group as the fundamental group $\pit(Y,y_0)$ of a space $Y$ obtained by attaching 2-cells to a generalized wedge (Theorem \ref{realizingtopgrps}). Finally, a topological van Kampen theorem (Theorem \ref{vankampentheorem}) enhances the computability of $\pi_{1}^{\tau}$ in terms of pushouts of topological groups. These three results compose the three parts of Section 4.
\section{Preliminaries and notation}
Prior to constructing $\pitxxo$, we recall a few basic constructions and facts. For spaces $X,Y$, let $M(X,Y)$ denote the set of continuous maps $X\ra Y$ with the compact open topology generated by subbasis sets $\langle K,U\rangle=\{f|f(K)\subset U\}$ where $K\subseteq X$ is compact and $U$ is open in $Y$. If $A\subseteq X$ and $B\subseteq Y$, let $M((X,A),(Y,B))$ denote the subspace of maps such that $f(A)\subseteq B$. If $A=\{x\}$ and $B=\{y\}$ are basepoints of $X$ and $Y$, we write $M_{\ast}(X,Y)$. If $f:Y\ra Z$ is a map, let $f_{\#}:M(X,Y)\ra M(X,Z)$, $k\mapsto f\circ k$ be the continuous map induced on mapping spaces. In particular, let $\px$ be the free path space $M(I,X)$ where $I=[0,1]$ is the unit interval. Clearly this construction results in a functor $P:\spaces\ra \spaces$, where $P(f)=f_{\#}$ on morphisms.

Suppose $X$ is a space and $\mathscr{B}_{X}$ is a basis for the topology of $X$ which is closed under finite intersection (for instance, the topology of $X$ itself). We are interested in using a convenient basis for the compact-open topology of $\px$. This basis consists of neighborhoods of the form $\bigcap_{j=1}^{n}\langle K_{n}^{j},U_j\rangle$ where $K_{n}^{j}=\left[\frac{j-1}{n},\frac{j}{n}\right]$ and $U_j\in\mathscr{B}_{X}$. We frequently use these neighborhoods since they are easy to manipulate and allow us to intuit basic open neighborhoods in $\px$ as finite, ordered sets of ``instructions."

The following notation for subspaces of $\px$ will be used: For $x_0,x_1\in X$, let $P(X,x_0)=\{\alpha\in \px|\alpha(0)=x_0\}$, $P(X,x_0,x_1)=\{\alpha\in \px|\alpha(i)=x_i,\text{ }i=0,1\}$, and $\Omega (X,x_0)=P(X,x_0,x_0)$. When the choice of basepoint is understood, we often write $\ox$ for the loop space. In notation, we will not always distinguish a neighborhood $\bigcap_{j=1}^{n}\langle C_j,U_j\rangle$ from being an open set in $\px$ or any of its subspaces. We say a loop $\alpha\in \Omega(X,x)$ is \textit{trivial} if it is homotopic to the constant loop $c_x$ at $x$ and \textit{non-trivial} if it is not trivial.

We make use of the following notation for restricted paths and neighborhoods as in \cite{Br10.1}. For any fixed, closed subinterval $A\subseteq I$, let $H_A:I\ra A$ be the unique, increasing, linear homeomorphism. For a path $p:I\ra X$, the \textit{restricted path of} $p$ \textit{to} $A$ is the composite $p_{A}=p|_{A}\circ H_{A}:I\ra A\ra X$. As a convention, if $A=\{t\}\subseteq I$ is a singleton, $p_{A}$ will denote the constant path at $p(t)$. Note that if $0=t_0\leq t_1\leq ...\leq t_n=1$, knowing the paths $p_{[t_{i-1},t_i]}$ for $i=1,...,n$ uniquely determines $p$. It is simple to describe concatenations of paths with this notation: If $p_1,...,p_n:I\ra X$ are paths such that $p_j(1)=p_{j+1}(0)$ for each $j=1,...,n-1$, then the \textit{n-fold concatenation} of this sequence of paths is the unique path $q=p_1\ast p_2\ast \dots\ast p_n$ such that $q_{K_{n}^{j}}=p_j$ for each $j=1,...,n$. It is a basic fact of the compact-open topology that concatenation $\px\times_{X}\px=\{(\alpha,\beta)|\alpha(1)=\beta(0)\}\ra \px$, $(\alpha,\beta)\mapsto \alpha\ast\beta$ is continuous. If $\alpha\in \px$, then $\alpha^{-1}(t)=\alpha(1-t)$ is the \textit{reverse} of $\alpha$ and for a set $A\subseteq \px$, $A^{-1}=\{\alpha^{-1}|\alpha\in A\}$. The operation $\alpha\mapsto \alpha^{-1}$ is a self-homeomorphism of $\px$.

Let $\mathcal{U}=\bigcap_{j=1}^{n}\langle C_j,U_j\rangle$ be an open neighborhood of a path $p\in \px$. Then $\mathcal{U}_{A}=\bigcap_{A\cap C_j\neq \emptyset}\langle H_{A}^{-1}(A\cap C_j),U_j\rangle$ is an open neighborhood of $p_{A}$ called the \textit{restricted neighborhood} of $\mathcal{U}$ to $A$. If $A=\{t\}$ is a singleton, then $\mathcal{U}_{A}=\bigcap_{t\in C_j}\langle I,U_j\rangle=\langle I,\bigcap_{t\in C_j}U_j\rangle$. On the other hand, if $p=q_{A}$ for some path $q\in \px$, then $\mathcal{U}^{A}=\bigcap_{j=1}^{n}\langle H_{A}(C_j),U_j\rangle$ is an open neighborhood of $q$ called the \textit{induced neighborhood} of $\mathcal{U}$ on $A$. If $A$ is a singleton so that $p_A$ is a constant map, let $\mathcal{U}^{A}=\bigcap_{j=1}^{n}\langle \{t\},U_j\rangle$. We frequently make use of the following Lemma which is straightforward to verify.
\begin{lemma} Let $\mathcal{U}=\bigcap_{j=1}^{n}\langle C_j,U_j\rangle$ be an open neighborhood in $\px$ such that $\bigcup_{j=1}^{n}C_j=I$. Then
\begin{enumerate}
\item For any closed interval $A\subseteq I$, $(\mathcal{U}^{A})_{A}=\mathcal{U}\subseteq (\mathcal{U}_{A})^{A}$
\item If $0=t_0\leq t_1\leq t_2\leq ...\leq t_n=1$, then $\mathcal{U}= \bigcap_{i=1}^{n}( \mathcal{U}_{[t_{i-1},t_i]})^{[t_{i-1},t_i]}.$
\end{enumerate}
\end{lemma}
\section{A group topology on the fundamental group}
\subsection{Free topological groups and the reflection $\tau$}
The \textit{free (Markov) topological group} on an unbased space $Y$ is the unique topological group $\fmy$ with a continuous map $\sigma:Y\ra \fmy$ universal in the sense that for any map $f:Y\ra G$ to a topological group $G$, there is a unique continuous homomorphism $\tilde{f}:\fmy\ra G$ such that $f=\tilde{f}\circ \sigma$. One can show that $\fmy$ exists by showing the forgetful functor $\tg\ra \spaces$ from the category of topological groups to topological spaces has a left adjoint $F_M:\spaces\ra \tg$. This is achieved by an application of Taut liftings or the General Adjoint Functor Theorem \cite{Po91}. The underlying group of $\fmy$ is simply the free group $F(Y)$ on the underlying set of $Y$ and $\sigma:Y\ra \fmy$ is the canonical injection of generators. The reader is referred to \cite{Th74} for proofs of the following basic facts.
\begin{lemma} \label{freetopgrpfacts} Let $X$ and $Y$ be topological spaces.
\begin{enumerate}
\item $\fmy$ is Hausdorff (discrete) if and only if $Y$ is functionally Hausdorff\footnote{A space is functionally Hausdorff if distinct points may be separated by continuous real valued functions.} (discrete).
\item The canonical map $\sigma:Y\ra \fmy$ is an embedding if and only if $Y$ is completely regular.
\item If $q:X\ra Y$ is a quotient map, then so is $F_{M}(q):\fmx\ra \fmy$.
\end{enumerate}
\end{lemma}

We use free topological groups to make the following useful construction:

A \textit{group with topology} is a group $G$ with a topology but where no restrictions are made on the continuity of the operations. The topology of $G$ will typically be denoted $\mathcal{T}_{G}$. Let $\gwt$ be the category of groups with topology and continuous homomorphisms. Given any $G\in\gwt$, the identity $G\ra G$ induces the multiplication epimorphism $m_G:F(G)\ra G$. Give $G$ the quotient topology with respect to $m_G:F_{M}(G)\ra G$ and denote the resulting group as $\tau(G)$. 

In general, if $p:H\ra G$ is an epimorphism of groups and $H$ is a topological group, $G$ becomes a topological group when it is given the quotient topology with respect to $p$. Therefore, $\tau(G)$ is a topological group. The identity $G\ra \tau(G)$ is continuous since it is the composition $m_G\circ \sigma:G\ra F_{M}(G)\ra \tau(G)$ and, moreover, $\tau(G)$ enjoys the universal property: \textit{If} $f:G\to H$ \textit{is any continuous homomorphism to a topological group} $H$, \textit{then} $f:\tau(G)\to H$ \textit{is continuous.}

Indeed, $f$ induces a continuous homomorphism $\tilde{f}:F_{M}(G)\ra H$ such that the diagram $$\xymatrix{ G \ar[dr]_-{id} \ar[r]^-{\sigma} & F_{M}(G) \ar[d]^-{m_G} \ar[dr]^-{\tilde{f}} \\ & \tau(G) \ar[r]_-{f} & H }$$commutes. Since $m_G$ is quotient, $f:\tau(G)\ra H$ is continuous.

Stated entirely in categorical terms this amounts to the fact that $\tg$ is a full reflective subcategory of $\gwt$.
\begin{lemma} \label{functoralityoftau}
$\tau:\gwt\ra \tg$ is a functor left adjoint to the inclusion functor $i:\tg\hookrightarrow \gwt$. Moreover, each reflection map $r_G:G\ra \tau(G)$ is the continuous identity homomorphism.
\end{lemma}
\begin{proof}
Define $\tau$ to be the identity on underlying groups and homomorphisms. Thus to see that $\tau$ is a well-defined functor, it suffices to check that $\tau(f):\tau(G)\to \tau(H)$ is continuous. Note that $F_{M}(f):F_{M}(G)\ra F_{M}(H)$ is a continuous homomorphism and the square $$\xymatrix{ F_M(G)  \ar[d]_{m_G} \ar[r]^{F_M(f)} & F_M(H) \ar[d]^{m_H} \\ \tau(G) \ar[r]_{\tau(f)} & \tau(H) }$$ of continuous homomorphisms commutes. Continuity of $\tau(f)$ follows from the fact that the left vertical map in this diagram is quotient.

The natural bijection characterizing the adjunction is $\tg(\tau(G),H)\cong \gwt(G,i(H))$, $f\mapsto f\circ r_G$ and follows from the universal property of $\tau(G)$ illustrated above.
\end{proof}
It is often convenient to think of $\tau$ as a functor which removes the smallest number of open sets from the topology of $G$ so that a topological group is obtained. 
\begin{proposition} \label{discoftaugrp}
The functor $\tau$ has the following properties.
\begin{enumerate}
\item $\tau$ preserves finite products.
\item $\tau$ preserves quotient maps.
\item If $G\in\gwt$, then $G$ is a topological group if and only if $G=\tau(G)$.
\item If $G\in\gwt$, then $G$ is discrete if and only if $\tau(G)$ is discrete.
\end{enumerate}
\end{proposition}
\begin{proof}\text{ }
\begin{enumerate}
\item To see that $\tau$ preserves finite products, let $G,H\in \gwt$. Clearly $G\times H$ with the product topology is the categorical product in $\gwt$. Apply $\tau$ to the projections of $G\times H$ to induce the continuous identity homomorphism $\tau(G\times H)\ra \tau(G)\times \tau(H)$. The inclusions $i:G\ra G\times H$, $i(g)=(g,e_H)$ and $j:H\ra G\times H$, $j(h)=(e_G,h)$ are embeddings of groups with topology. Let $\mu$ be the continuous multiplication of $\tau(G\times H)$. The composition $$\mu\circ (\tau(i)\times \tau(j)):\tau(G)\times \tau(H)\ra \tau(G\times H)\times \tau(G\times H)\ra \tau(G\times H)$$ is the continuous identity proving that $id:\tau(G\times H)\cong \tau(G)\times\tau(H)$.
\item Suppose $f:G\ra H$ is a homomorphism which is also a topological quotient map. Since $F_M$ preserves quotients, $F_{M}(f)$ is a quotient map. The top and vertical maps in the diagram of the proof of Lemma \ref{functoralityoftau} are all quotient. It follows that $\tau(f)$ is quotient.
\item One direction is obvious. If $G$ is a topological group, then the identity $id:G\ra G$ induces the continuous identity $\tau(G)\ra G$ which is the inverse of $r_G:G\ra \tau(G)$.
\item Since the identity $G\to \tau(G)$ is continuous, $G$ is discrete whenever $\tau(G)$ is. Conversely, if $G$ is discrete, then so is $F_M(G)$ and its quotient $\tau(G)$.
\end{enumerate}
\end{proof}
So far, we have only constructed $\tau(G)$ as the quotient of a free topological group. Explicit descriptions of free topological groups are known \cite{Sip05} but are, in general, quite complicated. For this reason, we provide an alternative approach to the topology of $\tau(G)$ when $G$ is a quasitopological group. We follow the well-known procedure of approximating group topologies through a transfinite process of taking quotient topologies \cite{CS90,Mal57}.

If $G$ is a quasitopological group, let $c(G)$ be the underlying group of $G$ with the quotient topology with respect to multiplication $\mu_{G}:G\times G\ra G$.
\begin{proposition} \label{qtg1}
$c:\qtg\ra \qtg$ is a functor when defined to be the identity on morphisms.
\end{proposition}
\begin{proof}
For $G\in\qtg$, consider the diagram $$\xymatrix{G\times G \ar[r] \ar[d]_-{\mu_{G}} & G\times G \ar[d]^-{\mu_{G}} \\ c(G) \ar[r] & c(G) }$$Let $g\in c(G)$ and the top map be $(a,b)\mapsto (b^{-1},a^{-1})$ (resp. $(a,b)\mapsto (ga,b)$, $(a,b)\mapsto (a,bg)$). Each of these functions are continuous since $G$ is a quasitopological group. The diagram commutes when the bottom map is inversion (resp. left multiplication by $g$, right multiplication by $g$). Since the vertical maps are quotient, the universal property of quotient spaces implies that these operations in $c(G)$ are continuous. A similar argument gives that $c(f)=f:c(G)\ra c(G')$ is continuous for each continuous homomorphism $f:G\ra G'$ of quasitopological groups.
\end{proof}
\begin{proposition} \label{ordinaltopologies}
Let $G$ be a quasitopological group.
\begin{enumerate}
\item The identity homomorphisms $G\ra c(G)\ra \tau(G)$ are continuous.
\item Then $\tau(c(G))=\tau(G)$.
\item Then $G$ is a topological group if and only if $G=c(G)$.
\end{enumerate}
\end{proposition}
\begin{proof}
\begin{enumerate}
\item Let $e$ be the identity of $G$. Consider the diagram $$\xymatrix{G\times \{e\} \ar[r]^{id\times e} \ar[d]_-{\mu_{G}}^-{\cong} & G\times G \ar[r]^-{r_G \times r_G} \ar[d]_-{\mu_{G}} & \tau(G)\times \tau(G) \ar[d]_-{\mu_{G}} \\ G \ar[r]_-{id} & c(G) \ar[r]_-{id} & \tau(G)}$$Each vertical map is quotient and the maps in the top row are continuous. The identities in the bottom row are continuous by the universal property of quotient spaces.
\item Applying $\tau$ to $id:G\ra c(G)$ gives $id:\tau(G)\ra \tau(c(G))$. The adjoint of $c(G)\ra \tau(G)$ is the inverse $\tau(c(G))\ra \tau(G)$. This gives the equality $\tau(c(G))=\tau(G)$.
\item If $G$ is a topological group, then $G=\tau(G)$ and all three topologies on $G$ agree by the first statement. Conversely, if $G=c(G)\in \qtg$, then $\mu_{G}:G\times G\ra c(G)=G$ is continuous and $G$ is a topological group.
\end{enumerate}
\end{proof}
The proof of the next proposition is a straightforward exercise left to the reader.
\begin{proposition} \label{qtg2}
If $G_{\lambda}$ is a family of quasitopological groups each with underlying group $G$ and topology $\mathcal{T}_{G_{\lambda}}$, then the topology $\bigcap_{\lambda}\mathcal{T}_{G_{\lambda}}$ on $G$ makes $G$ a quasitopological group.
\end{proposition}
\begin{approximation} \label{approximation} \emph{
Let $G=G_{0}$ be a quasitopological group with topology $\mathcal{T}_{G_0}$. We iterate the above construction by letting $G_{\alpha}=c(G_{\alpha-1})$ (with topology $\mathcal{T}_{G_{\alpha}}$) for each ordinal $\alpha$ with a predecessor. When $\alpha$ is a limit ordinal, let $G_{\alpha}$ have topology $\mathcal{T}_{G_{\alpha}}=\bigcap_{\beta<\alpha}\mathcal{T}_{G_{\beta}}$. Applying Propositions \ref{qtg1} and \ref{qtg2} in a simple transfinite induction argument gives that $G_{\alpha}$ is a quasitopological group for each ordinal $\alpha$. Using a standard cardinality argument, we now show the topologies $\mathcal{T}_{G_{\alpha}}$ stabilize to the topology $\mathcal{T}_{\tau(G)}$ of $\tau(G)$.}
\end{approximation}
\begin{theorem}
There is an ordinal $\alpha$ such that $G_{\alpha}=\tau(G)$.
\end{theorem}
\begin{proof}
We have already noted that $\mathcal{T}_{\tau(G)}\subseteq \mathcal{T}_{G_{\alpha+1}}\subseteq  \mathcal{T}_{G_{\alpha}}\subseteq \mathcal{T}_{G_0}$ for every $\alpha$. Since the identity homomorphisms $G\ra G_{\alpha}\to \tau(G)$ are continuous, $G_{\alpha}=\tau(G)$ whenever $G_{\alpha}$ is a topological group. Therefore, if $G_{\alpha}\neq \tau(G)$, then $G_{\alpha}$ is not a topological group and Prop. \ref{ordinaltopologies} implies that $\mathcal{T}_{G_{\alpha+1}}$ is a proper subset of $\mathcal{T}_{G_{\alpha}}$. Consequently, if $G_{\alpha}\neq\tau(G)$ for every ordinal $\alpha$, there are distinct sets $U_{\alpha}\in \mathcal{T}_{G_{\alpha}}-\mathcal{T}_{G_{\alpha+1}}\subseteq \mathcal{T}_{G_0}$ indexed by the ordinals. This contradicts the fact that there is an ordinal which does not inject into the power set of $\mathcal{T}_{G_0}$.
\end{proof}
Since the identity homomorphisms $G_{\beta}\ra G_{\alpha}\ra \tau(G)$ are continuous whenever $\beta\leq\alpha$, the groups $G_{\alpha}$ stabilize to $\tau(G)$.
\begin{corollary} \label{samesubgroups}
If $G\in\qtg$, then $G$ and $\tau(G)$ have the same open subgroups.
\end{corollary}
\begin{proof}
Since the identity $G\ra \tau(G)$ is continuous, every open subgroup of $\tau(G)$ is open in $G$. For the converse, we check that every open subgroup of $G$ is open in $c(G)$ and proceed by transfinite induction.

Suppose $H$ is an open subgroup of $G$. Note that if $\mu_G:G\times G\ra G$ is multiplication, then $\mu_{G}^{-1}(H)=\bigcup_{ab\in H}Hb^{-1}\times a^{-1}H$. Since $G$ is a quasitopological group and $H$ is open in $G$, $\mu_{G}^{-1}(H)$ is open in $G\times G$. By construction, $\mu_{G}:G\times G\ra c(G)$ is quotient and therefore $H$ is open in $c(G)$.

Suppose $H$ is an open subgroup of $G=G_0$. Thus $H\in \mathcal{T}_{G_{0}}$. If $H\in \mathcal{T}_{G_{\beta}}$ for each $\beta<\alpha$, then certainly $H\in \mathcal{T}_{G_{\alpha}}=\bigcap_{\beta<\alpha}\mathcal{T}_{G_{\beta}}$ if $\alpha$ is a limit ordinal. If $\alpha$ is a successor ordinal, then $H\in \mathcal{T}_{G_{\alpha-1}}$ and the previous paragraph gives that $H\in \mathcal{T}_{c(G_{\alpha-1})}=\mathcal{T}_{G_{\alpha}}$. Therefore $H$ is open in $G_{\alpha}$ for every ordinal $\alpha$. Since the $G_{\alpha}$ stabilize to $\tau(G)$, $H$ is open in $\tau(G)$.
\end{proof}
%

The category $\tg$ is complete and cocomplete and the underlying group of the colimit/limit of a diagram $J\ra \tg$ agrees with the limit/colimit of the diagram $J\ra \tg\ra \grp$ of underlying groups. While limits in $\tg$ have obvious descriptions, the topology of a colimit can (as free topological groups) be quite complicated. The topological van Kampen theorem presented in Section 4.3 motivates a brief note on the relationship between $\tau$ and pushouts in $\tg$. Let $G_1\ast_{G} G_2$ denote the pushout (or \textit{free topological product with amalgamation}) of a diagram $G_1 \leftarrow G \ra G_2$. If $G=\{\ast\}$, this is simply the \textit{free topological product} $G_1\ast G_2$. The topology of $G_1\ast G_2$ is quotient with respect to the projection $k_{G_1,G_2}:F_{M}(G_1\sqcup G_2)\ra G_1\ast G_2$ (where $G_1\sqcup G_2$ is the coproduct in $\spaces$) and $G_1\ast_{G} G_2$ is quotient with respect to the projection $G_1\ast G_2\ra G_1\ast_{G} G_2$.
\begin{proposition} \label{coproductsandtau}
For groups with topology $G_1,G_2$, the canonical epimorphism $k_{G_1,G_2}:F_{M}(G_1\sqcup G_2)\ra \tau(G_1)\ast\tau(G_2)$ is a topological quotient map.
\end{proposition}
\begin{proof} The following diagram commutes in the category of topological groups.
$$\xymatrix{F_{M}(F_{M}(G_1)\sqcup F_{M}(G_2)) \ar[d]_-{F_{M}\left(m_{G_1}\sqcup m_{G_2}\right)} \ar[rr]^-{k_{F_{M}(G_1),F_{M}(G_2)}} && F_{M}(G_1)\ast F_{M}(G_2) \ar[r]^-{\cong} &  F_{M}(G_1\sqcup G_2) \ar[d]^-{k_{G_1,G_2}} \\
F_{M}(\tau(G_1)\sqcup \tau(G_2)) \ar[rrr]_-{k_{G_1,G_2}} &&&  \tau(G_1)\ast \tau(G_2)
}$$Since $m_{G_1},m_{G_2}$ are quotient and $F_M$ preserves quotients, all maps except for the right vertical map are known to be quotient. By the universal property of quotient spaces, $k_{G_1,G_2}:F_{M}(G_1\sqcup G_2)\ra \tau(G_1)\ast\tau(G_2)$ must also be quotient.
\end{proof}
\subsection{The construction and characterizations of $\pitxxo$}
The \textit{path component space} of a space $X$, denoted  $\piztop$, is the set of path components $\pi_{0}(X)$ with the quotient topology with respect to the canonical function $X\ra \pi_{0}(X)$ identifying path components. Since a map $X\to Y$ induces a continuous map $\piztop\to \pi_{0}^{qtop}(Y)$ on path component spaces, we obtain a functor $\pi_{0}^{qtop}:\spaces\to \spaces$.

The path component space $\pi_{1}^{qtop}(X,x_0)=\pi_{0}^{qtop}(\Omega(X,x_0))$ is the \textit{quasitopological fundamental group} of $(X,x_0)$. Since multiplication and inversion in the fundamental group are induced by the continuous operations $(\alpha,\beta)\mapsto \alpha\ast\beta$ and $\alpha\to \alpha^{-1}$ in the loop space, it is immediate from the universal property of quotient spaces that $\pi_{1}^{qtop}(X,x_0)$ is a quasitopological group.

Since a based map $f:(X,x_0)\to (Y,y_0)$ induces a continuous group homomorphism $f_{\ast}=\pi_{0}^{qtop}(\Omega(f)):\pi_{1}^{qtop}(X,x_0)\to\pi_{1}^{qtop}(Y,y_0)$, this construction results in a functor $\pi_{1}^{qtop}:\hbspaces\ra \qtg$ on the homotopy category of based spaces. It is worth noting that the isomorphism class of the quasitopological fundamental group is independent of the choice of basepoint (See the proof of Proposition \ref{changingbasepoints} below). We write $\pitopx$ when the choice of basepoint is understood.

\begin{construction}
\emph{
As mentioned in the introduction, it is know that $\pi_{1}^{qtop}(X,x_0)$ fails to be a topological group even for compact planar sets, however, since $\qtg$ is a full subcategory of $\gwt$, the composition of functors $\pi_{1}^{\tau}= \tau\circ \pi_{1}^{qtop}:\hbspaces\ra \tg$ is well-defined. This new functor assigns, to a based space $X$, a topological group $\pitx$ whose underlying group is $\pi_{1}(X)$. Since the identity homomorphism $\pitopx\to \pitx$ is continuous, so is \[\pi:\Omega(X)\to \pitopx \to\pitx.\]
}
\end{construction}

The topological group $\pitx$ can be characterized in a number of ways: According to the explicit construction of $\tau$, $\pitx$ is the quotient of the free topological group $F_M(\pitopx)$. Additionally, since $F_M$ preserves quotient maps, $F_M(\pi):F_{M}(\ox)\to F_M(\pitopx)$ is quotient and thus $\pitx$ is the quotient of $F_{M}(\ox)$ with respect to the map $\alpha_1...\alpha_n\mapsto [\alpha_1\ast\dots\ast\alpha_n]$. It is also convenient to characterize $\pitx$ in terms of its universal property.
\begin{proposition} \label{finesttopology}
If $\Phi:\pitx\ra G$ is a homomorphism to a topological group $G$ such that $\Phi\circ \pi:\ox\ra \pitx\ra G$ is continuous, then $\Phi$ is also continuous.
\end{proposition}
\begin{proof}
If $\Phi\circ \pi:\ox\ra G$ is continuous, then $\Phi:\pitopx\to G$ is continuous by the universal property of quotient spaces. Since $G$ is a topological group, the adjoint $\Phi:\pitx\to G$ is continuous.
\end{proof}
\begin{theorem}
The topology of $\pitx$ is the finest group topology on $\pi_1(X)$ such that $\pi:\ox\ra \pi_{1}(X)$ is continuous.
\end{theorem}
\begin{proof}
Suppose $G$ is a topological group with underlying group $\pi_{1}(X)$ and is such that $\pi:\ox\to G$ is continuous. Since $id\circ \pi:\ox\to \pitx\to G$ is continuous, $id:\pitx\to G$ is continuous by Proposition \ref{finesttopology}. Thus the topology of $\pitx$ is finer than that of $G$.
\end{proof}
\begin{remark}
\emph{
Fabel \cite{Fab11} has recently shown that for each $n\geq 2$ the quasitopological homotopy group $\pi_{n}^{qtop}=\pi_{1}^{qtop}\circ \Omega^{n-1}$, first studied in \cite{GHMM08,GH09,GHMM10}, fails to take values in the category of topological abelian groups. We can then define $\pi_{n}^{\tau}=\tau\circ \pi_{n}^{qtop}=\pi_{1}^{\tau}\circ \Omega^{n-1}$ to obtain a truly ``topological" higher homotopy group. In the present paper, we do not give any special attention to these higher homotopy groups.
}
\end{remark}
%
%
%
\subsection{Basic properties of $\pitxxo$}
Since $\pit$ is defined as the composition $\tau\circ \pi_{1}^{qtop}$, it is often practical to approach the topology of $\pitx$ by studying the quotient topology and the behavior of $\tau$ separately. For instance, if $h:\pitopy\cong\pitopx$ as quasitopological groups, then $\tau(h):\pitx\cong\pity$ as topological groups. In this way, the homotopy invariance of $\pit$ follows from the homotopy invariance of $\pi_{1}^{qtop}$.

The following proposition indicates the isomorphism class of $\pitx$ (for path connected $X$) is independent of the choice of basepoint.
\begin{proposition} \label{changingbasepoints}
If $\gamma:I\ra X$ is a path, then $\pi_{1}^{\tau}(X,\gamma(1))\ra  \pi_{1}^{\tau}(X,\gamma(0))$, $[\alpha]\mapsto [\gamma\ast\alpha\ast\gamma^{-1}]$ is an isomorphism of topological groups. Consequently, if $x_0,x_1$ lie in the same path component of $X$, then $\pi_{1}^{\tau}(X,x_0)\cong \pi_{1}^{\tau}(X,x_1)$.
\end{proposition}
\begin{proof}
The continuous operation $c_{\gamma}:\Omega(X,\gamma(1))\to \Omega(X,\gamma(0))$ given by $\alpha\mapsto \gamma\ast\alpha\ast\gamma^{-1}$ induces the continuous, group isomorphism $\Gamma:\pi_{1}^{qtop}(X,\gamma(1))\ra  \pi_{1}^{qtop}(X,\gamma(0))$, $[\alpha]\mapsto [\gamma\ast\alpha\ast\gamma^{-1}]$ on path component spaces. The inverse is continuous since $c_{\gamma^{-1}}$ is continuous. Thus $\Gamma$ is an isomorphism of quasitopological groups and $\tau(\Gamma)$ is an isomorphism of topological groups.
\end{proof}
Recall from 3. of Proposition \ref{discoftaugrp} that $\pitx=\pitopx$ if and only if $\pitopx$ is a topological group. Thus if $\pitopx$ fails to be a topological group, the topology of $\pitx$ is strictly coarser than that of $\pitopx$. Despite this fact that some open sets may ``be removed" from the quotient topology on $\pi_1(X)$ by applying $\tau$, these two groups always share the same open subgroups.
\begin{proposition}
For any based space $X$, $\pitopx$ and $\pitx$ have the same open subgroups.
\end{proposition}
\begin{proof}
This is a special case of Corollary \ref{samesubgroups}.
\end{proof}
The following characterization of $X$ for which $\pitx$ is a discrete group reinforce the idea that the topology of $\pitx$ is defined to remember only non-trivial, local homotopical properties of spaces. 
\begin{proposition} \label{discretenessgeneral} For any path connected space $X$, the following are equivalent:
\begin{enumerate}
\item $\pitx$ is a discrete group.
\item $\pitopx$ is a discrete group.
\item For every null-homotopic loop $\alpha\in \ox$, there is an open neighborhood $\mathcal{U}$ of $\alpha$ in $\ox$ containing only null-homotopic loops.
\end{enumerate}
\end{proposition}
\begin{proof}
1. $\Leftrightarrow$ 2. is a special case of the second part of Proposition \ref{discoftaugrp}. 2. $\Leftrightarrow$ 3. follows from the definition of the quotient topology.
\end{proof}
It is more convenient to characterize discreteness in terms of local properties of $X$ itself by applying the characterization of discreteness of the quasitopological fundamental group in \cite{CM}.
\begin{corollary} \label{discreteness}
Suppose $X$ is path connected. If $\pitx$ is discrete, then $X$ is semilocally 1-connected. If $X$ is locally path connected and semilocally 1-connected, then $\pitx$ is discrete.
\end{corollary}

Thus $\pitx$ is discrete when $X$ has the homotopy type of a CW-complex, manifold, or, more generally, any locally path connected space with a universal covering space. 

Since $\pitopx$ is not always a topological group, $\pi_{1}^{qtop}$ does not preserve finite products. The following proposition illustrates a first advantage of $\pit$. Even though it is not a direct consequence of part 1. of Proposition \ref{discoftaugrp}, the proof is essentially the same.
\begin{proposition} \label{product}
For any based spaces $X,Y$, there is a natural isomorphism $\phi:\pi_{1}^{\tau}(X\times Y)\ra \pitx\times\pity$ of topological groups.
\end{proposition}
\begin{proof}
The projections $X\times Y\ra X$ and $X\times Y\ra Y$ induce the continuous group isomorphism $\phi:\pi_{1}^{\tau}(X\times Y)\ra\pitx\times\pity$ given by $\phi([(\alpha,\beta)])=([\alpha],[\beta])$. Let $x_0$ and $y_0$ be the basepoints of $X$ and $Y$ respectively and $c_{x_0}:I\ra X$ and $c_{y_0}:I\ra Y$ be the constant loops. The maps $i:X\ra X\times Y$, $x\mapsto (x,y_0)$ and $j:Y\ra X\times Y$, $x\mapsto (x_0,y)$ induce the continuous homomorphisms $i_{\ast}:\pitx\ra \pi_{1}^{\tau}(X\times Y)$, $i_{\ast}([\alpha])=[(\alpha,c_{y_0})]$ and $j_{\ast}:\pity\ra \pi_{1}^{\tau}(X\times Y)$, $j_{\ast}([\beta])=[(c_{x_0}, \beta)]$. Let $\mu$ be group multiplication of the topological group $\pi_{1}^{\tau}(X\times Y)$. The composition $\mu\circ (i_{\ast}\times j_{\ast})$ is the continuous inverse of $\phi$.
\end{proof}
The author does not know if $\pi_{1}^{\tau}$ preserves arbitrary products, however, there is a positive answer when the factor spaces have discrete fundamental groups.
\begin{example} \label{exampleproductsofcwcomplexes} \emph{
Let $X_{n}$, $n\geq 1$ be a countable family of spaces such that $\pi_{1}^{\tau}(X_n)$ is discrete (e.g. $X_n$ a polyhedron or manifold) for each $n$. Since the quotient map $\pi_{n}:\Omega(X_{n})\ra \pi_{1}^{qtop}(X_{n})=\pi_{1}^{\tau}(X_{n})$ is open for each $n\geq 1$, both vertical maps in the diagram$$\xymatrix{ \ox \ar[rr]^-{\cong} \ar[d]_-{\pi} && \prod_{n\geq 1}\Omega(X_{n}) \ar[d]^-{\prod_{n\geq 1} \pi_{n}} \\ \pitopx=\pitx \ar[rr]_-{\cong} && \prod_{n\geq 1}\pi_{1}^{qtop}(X_{n})= \prod_{n\geq 1}\pi_{1}^{\tau}(X_{n}) }$$ are quotient. The top (resp. bottom) map is the canonical homeomorphism (resp. group isomorphism). It follows from the universal property of quotient spaces that $\pitx$ is a countable product of discrete groups and is therefore a zero-dimensional, metrizable topological group. Moreover, if $\pi_{1}(X_n)\neq 1$ for infinitely many $n$, $\pitx$ is not discrete. For instance, $\pi_{1}^{\tau}\left(\prod_{n\geq 1}S^1\right)$ is isomorphic to the countable product $\prod_{n\geq 1}\mathbb{Z}$.
}
\end{example}
\subsection{Separation and the first shape group}
Characterizing the spaces $X$ for which $\pitx$ is Hausdorff is a non-trivial task. Since every Hausdorff topological group is functionally Hausdorff, the continuity of the identity $\pitopx\ra \pitx$ offers an obvious necessary condition:
\begin{proposition} \label{hausdorfftofunctionallyhausdorff}
If $\pitx$ is Hausdorff, then $\pitopx$ is functionally Hausdorff.
\end{proposition}
The $T_1$ axiom in $\pitopx$ is thus necessary for $\pitx$ to be Hausdorff but cannot be sufficient since there are spaces $Y$ such that $\pi_{1}^{qtop}(Y)$ is $T_1$ but not functionally Hausdorff \cite[Example 4.13]{Br10.1}. The converse of Proposition \ref{hausdorfftofunctionallyhausdorff} is true in the special case of Corollary \ref{hausdorffpitsus}. It remains an open question whether or not the converse holds in full generality. 

In the pursuit of practical necessary conditions for the existence of separation axioms, we recall the notion of a ``homotopically path-Hausdorff space" introduced in \cite{FRVZ11}. This notion is stronger than the notion of ``homotopically Hausdorff space" useful for studying generalized universal coverings of locally path connected spaces \cite{FZ07}.
\begin{definition} \emph{
A space $X$ is \textit{homotopically path-Hausdorff} if given any paths $\alpha,\beta:[0,1]\to X$ such that $\alpha(0)=\beta(0)$, $\alpha(1)=\beta(1)$, and $\alpha$ and $\beta$ are not homotopic rel. endpoints, then there exist a partition $0=t_0<t_1<t_2<...<t_k=1$ and open sets $U_1,U_2,...U_k\subseteq X$ with $\alpha([t_{i-1},t_{i}])\subset U_i$ such that for any other path $\gamma:[0,1]\to X$ satisfying $\gamma(t_{i})=\alpha(t_{i})$ for $i=0,...,k$ and $\gamma([t_{i-1},t_{i}])\subset U_i$ for $i=1,...,k$, $\gamma$ is not homotopic to $\beta$ rel. endpoints.}
\end{definition}
%
%
The following Proposition gives that the homotopically path-Hausdorff property is necessary for the Hausdorff separation axiom in our topologized fundamental group. The homotopically path-Hausdorff property is even more closely related to the $T_1$ axiom in quasitopological fundamental groups; this connection, studied by Paul Fabel and the author, will be explored further in a future paper.
\begin{proposition} \label{haushomotopyhaus}
If $X$ is path connected and $\pitx$ is Hausdorff, then $X$ is homotopically path-Hausdorff.
\end{proposition}
\begin{proof}
Suppose $\pitx$ is Hausdorff. Proposition \ref{changingbasepoints} gives that $\pit(X,x)$ is Hausdorff for all $x\in X$. Suppose $\alpha,\beta:[0,1]\to X$ are paths such that $\alpha(0)=\beta(0)$, $\alpha(1)=\beta(1)$, but which are not homotopic rel. endpoints. Thus $L=\alpha\ast\beta^{-1}$ is a non-trivial loop based at $\alpha(0)$.

Since $\pit(X,\alpha(0))$ is Hausdorff, there is an open neighborhood $U$ of $[L]$ which does not contain the identity $[c_{\alpha(0)}]$. Since $\pi:\Omega(X,\alpha(0))\to \pit(X,\alpha(0))$ is continuous, $\pi^{-1}(U)$ is an open neighborhood of $L$. Take a positive even integer $n=2k$ and open neighborhoods $U_1,...,U_n$ such that $\mathcal{U}=\bigcap_{j=1}^{n}\left\langle \left[\frac{j-1}{n},\frac{j}{n}\right],U_j\right\rangle$ is an open neighborhood of $L$ contained in $\pi^{-1}(U)$. Clearly $\mathcal{U}$ contains only non-trivial loops.

Let $t_i=\frac{i}{k}$ for each $i=0,1,...,k$ and note that $\alpha([t_{i-1},t_{i}])\subset U_i$ for $i=1,...,k$. Suppose $\gamma:[0,1]\to X$ is a path satisfying $\gamma(t_{i})=\alpha(t_{i})$ for each $i=0,...,k$ and $\gamma([t_{i-1},t_{i}])\subset U_i$ for $i=1,...,k$. Now the concatenation $\gamma\ast\beta^{-1}$ lies in $\mathcal{U}$ and therefore must be non-trivial. Thus $\gamma$ and $\beta$ are not homotopical relative to their endpoints.
\end{proof}

In the search for conditions sufficient to guarantee separation, we turn to shape theory. One is often interested in whether or not the fundamental group of a space $X$ with complicated local structure naturally injects into its first shape homotopy group since such an embedding provides a convenient way to understand the fundamental group. It turns out that this property is strong enough to guarantee that $\pitx$ is Hausdorff. We refer to \cite{MS82} for the preliminaries of shape theory.
\begin{Topologicalshapegroups} \label{topshapegroup} \emph{
One way to see that $\pitx$ is very often Hausdorff is to use the natural topology on the first shape homotopy group. The homotopy category of polyhedra $\mathbf{hPol_{\ast}}$ is the full-subcategory of $\hbspaces$ consisting of spaces with the homotopy type of a polyhedron. It is well-known that $\mathbf{hPol_{\ast}}$ is a dense subcategory of $\hbspaces$ \cite[\S 4.3, Theorem 7]{MS82}. This means that for every based space $X$, there is an $\mathbf{hPol_{\ast}}$-expansion $X\ra(X_{\lambda},p_{\lambda \lambda '},\Lambda)$ (for instance, the \u{C}ech expansion). Specifically, the expansion consists of polyhedra $X_{\lambda}, \lambda\in \Lambda$ and maps $p_{\lambda}:X\ra X_{\lambda}$ such that $p_{\lambda}$ is homotopic to $p_{\lambda \lambda '}\circ p_{\lambda '}$ whenever $\lambda '\geq \lambda$ in the directed set $\Lambda$. The induced continuous homomorphisms $(p_{\lambda})_{\ast}:\pi_{1}^{\tau}(X)\ra \pi_{1}^{\tau}(X_{\lambda})$ satisfy $(p_{\lambda})_{\ast}=(p_{\lambda \lambda '})_{\ast}\circ(p_{\lambda '})_{\ast}$.\\
\indent The \textit{first homotopy pro-group} of a based space is the inverse system $pro$-$\pi_{1}(X)=\left(\pi_{1}^{\tau}(X_{\lambda}),(p_{\lambda \lambda '})_{\ast},\Lambda\right)$ of discrete groups in $\mathbf{pro}$-$\tg$ where the bonding maps are the homomorphisms $(p_{\lambda \lambda '})_{\ast}:\pi_{n}^{\tau}(X_{\lambda '})\ra \pi_{1}^{\tau}(X_{\lambda})$. The \textit{first shape homotopy group} of $X$ is the limit $\breve{\pi}_{1}(X)=\varprojlim pro$-$\pi_{1}(X)$ which, as an inverse limit of discrete groups, is a Hausdorff topological group. The continuous homomorphisms $(p_{\lambda})_{\ast}$ induce a natural, continuous homomorphism $\phi:\pi_{1}^{\tau}(X)\ra \breve{\pi}_{1}(X)$.\\
\indent When $\phi$ is injective, $X$ is said to be $\mathit{\pi_1}$\textit{-injective}. This property is particularly useful since the fundamental group of a $\mathit{\pi_1}$-injective space can be identified as a subgroup of $\breve{\pi}_{1}(X)$. Some recent results on $\pi_1$-injectivity include \cite{CC06,Eda98,FZ05,FG05}.
}
\end{Topologicalshapegroups}
\begin{proposition} \label{injtohaus}
If $X$ is $\pi_1$-injective, then $\pitx$ is Hausdorff.
\end{proposition}
\begin{proof}
Any topological group continuously injecting into a Hausdorff group is Hausdorff.
\end{proof}
The converse of Proposition \ref{injtohaus} is false in general: See \cite[Example 2.4]{FG05} for a simple counterexample. It remains open whether or not the converse holds for locally path connected spaces.
\begin{example} \label{he} \emph{
For each integer $n\geq 1$, let $C_n=\left\{(x,y)\in\mathbb{R}^{2}|\left(x-\frac{1}{n}\right)^{2}+y^2=\frac{1}{n^2}\right\}$ so that $\mathbb{HE}= \bigcup_{n \geq 1}C_n$ is the usual Hawaiian earring. The shape group $\breve{\pi}_{1}(\mathbb{HE})$ is the inverse limit $\varprojlim_{n}F_n$ of discrete free groups; here $F_n$ is the free topological group on the discrete space of cardinality $n$. The canonical homomorphism $\phi:\pi_{1}(\mathbb{HE})\ra \varprojlim_{n}F_n$ is injective \cite{MM86} and therefore $\pi_{1}^{\tau}(\mathbb{HE})$ is Hausdorff. Since $\pi_{1}^{\tau}(\mathbb{HE})$ is a topological group, it is necessarily Tychonoff. Currently, it is not even known if the quasitopological fundamental group $\pi_{1}^{qtop}(\mathbb{HE})$ is regular.\\
\indent Though the homomorphism $\phi$ is injective, it is known that $\phi:\pi_{1}^{qtop}(\mathbb{HE})\to \varprojlim_{n}F_n$ fails to be a topological embedding \cite{Fab05.3}. Since $\pi_{1}^{qtop}(\mathbb{HE})$ is not a topological group \cite{Fab10}, the topology of $\pi_{1}^{\tau}(\mathbb{HE})$ is strictly coarser than the quotient topology of $\pi_{1}^{qtop}(\mathbb{HE})$. From this fact alone, it is plausible that $\phi:\pi_{1}^{\tau}(\mathbb{HE})\ra \varprojlim_{n}F_n$ is an embedding.\\
\indent The semicovering theory \cite{Brsemi} of $\mathbb{HE}$, however, can be used to show that $\phi:\pi_{1}^{\tau}(\mathbb{HE})\ra \varprojlim_{n}F_n$ also fails to be an embedding. Indeed, if this map is an embedding, the normal, open subgroups $K_n=\ker\left(\pi_{1}^{\tau}(\mathbb{HE})\to F_n\right)$ form a neighborhood base at the identity of $\pi_{1}^{\tau}(\mathbb{HE})$. A simple consequence of such a basis is: An subgroup $H$ is open in $\pi_{1}^{\tau}(\mathbb{HE})$ if and only if there is a covering map $p:Y\to X$ such that $p_{\ast}(\pi_1(Y))=H$. But this cannot be the case since open subgroups of $\pi_{1}^{\tau}(\mathbb{HE})$ are classified by semicoverings of $\mathbb{HE}$ and there are examples of semicovering maps $p:Y\to \mathbb{HE}$ which are not covering maps \cite[Example 3.8]{Brsemi}.\\
\indent This means we have three distinct natural topologies on $\pi_{1}(\mathbb{HE})$ from finest to coarsest: the quotient topology, the topology of $\pi_{1}^{\tau}(\mathbb{HE})$, and the initial topology with respect to $\phi:\pi_{1}(\mathbb{HE})\ra \varprojlim_{n}F_n$.
}
\end{example}
\begin{example}\emph{
The map $\phi:\pi_{1}^{\tau}(X)\ra \breve{\pi}_{1}(X)$ also fails to be an embedding in simple non-locally path connected cases. Consider the generalized wedge $\sus$ in Example \ref{exampleconvergingcircles} below where $X$ is the one-point compactification of the natural numbers. The free topological group $\pitsus\cong F_{M}(X)$ is not first countable \cite[7.1.20]{AT08} but since $\sus$ is a planar continuum, the shape group $\breve{\pi}_{1}(\sus)$ is metrizable and $\phi:\pi_{1}(\sus)\ra \breve{\pi}_{1}(\sus)$ is injective \cite{FZ05}. Therefore the topology of $\pitsus$ is strictly finer than the initial topology with respect to $\phi$.}
\end{example}
Despite the failure of embedding in the previous two examples, there are simple examples of non-semilocally simply connected spaces whose fundamental group embeds in its first shape group. For instance, if $X=\prod_{n\geq 1}S^1$ is the countable product of circles, as in Example \ref{exampleproductsofcwcomplexes}, then $\phi:\pi_{1}^{\tau}(X)\to\breve{\pi}_{1}(X)$ is an isomorphism of non-discrete metrizable groups. These examples motivate the following general question.
\begin{question} \label{question} \emph{
For which spaces $X$ is $\phi:\pitx\ra \breve{\pi}_{1}(X)$ a topological embedding?}
\end{question}
\section{Main Results: Realizing topological groups as fundamental groups}
The fact that one can conveniently realize a free group or a free product of groups as the fundamental group of a space is quite useful, for instance, in proving the Nielsen\textendash Schreier theorem and Kurosh subgroup theorem. The main results of the current paper, laid out in the following three sections, are meant to be such ``realization results" for topological groups.
\subsection{Generalized wedges of circles and free topological groups} \label{generalizedwedges}
The ability of $\pi_{1}^{\tau}$ to distinguish spaces with isomorphic fundamental groups can be observed by making a computation analogous to the elementary fact that the fundamental group of a wedge of circles indexed by a set $X$ is free on $X$. The \textit{generalized wedge of circles on (an unbased space)} $\mathit{X}$ is constructed as the reduced suspension $\sus$ of the space $X_+=X\sqcup \{\ast\}$ with isolated basepoint. Equivalently, it is the quotient space $X\times I/X\times \{0,1\}$. The image of $(x,t)\in X\times I$ in the quotient is written as $x\wedge t$. Similarly, $A\wedge B$ denotes the image of $A\times B\subset X\times I$. The canonical basepoint is $x_0=X\wedge \{0,1\}$. Since $\sus\cong \bigvee_{X}S^1$ whenever $X$ is a discrete space, this construction clearly generalizes the construction of a wedge of circles.

The main result in this section is that the fundamental group of a generalized wedge of circles $\sus$ is the free topological group on the path component space of $X$. The proof calls upon the technical computation of $\pitop$ in \cite{Br10.1}.
\begin{theorem} \label{theoremfreetopologicalgroups}
There is a natural isomorphism $h_X:F_{M}(\piztop)\ra \pitsus$ of topological groups.
\end{theorem}
\begin{proof}
The unit $u:X\ra \Omega(\sus)$ of the adjunction $\bspaces(\sus,Y)\cong \spaces(X,\Omega Y)$ induces a continuous injection $u_{\ast}:\piztop\ra \pitop$ on path component spaces. One of the main results in \cite{Br10.1} is that for arbitrary $X$, $u_{\ast}$ induces a natural group isomorphism $h_{X}:F(\piz)\ra \pi_{1}(\sus)$ so that $h_{X}^{-1}\circ u_{\ast}$ is the canonical injection of generators and $h_{X}^{-1}:\pitop\ra F_{M}(\piztop)$ is continuous. Since $F_{M}(\piztop)$ is a topological group, $h_{X}^{-1}:\pitsus\ra F_{M}(\piztop)$ is continuous. Additionally, the continuous injection $u_{\ast}:\piztop\ra \pitop\ra \pitsus$ induces (by the universal property of free topological groups) the continuous inverse $h_{X}:F_{M}(\piztop)\ra \pitsus$. 
\end{proof}
Prior to studying some special cases of generalized wedges, we apply the theory of free topological groups in two Corollaries.
\begin{corollary}
A quotient map (resp. homotopy equivalence) $f:X\ra Y$ induces a quotient map (resp. an isomorphism) $f_{\ast}:\pitsus\ra \pi_{1}^{\tau}(\Sigma(Y_+))$ of topological groups. 
\end{corollary}
\begin{proof}
This follows directly from the fact that both the functors $F_M$ and $\pi_{0}^{top}$ preserve quotients and that $\pi_{0}^{top}$ takes homotopy equivalences to homeomorphisms.
\end{proof}
\begin{corollary} \label{hausdorffpitsus}
For any unbased space $X$, the following are equivalent:
\begin{enumerate}
\item $\pitsus$ is Hausdorff.
\item $\pitop$ is functionally Hausdorff.
\item $\piztop$ is functionally Hausdorff.
\end{enumerate}
\end{corollary}
\begin{proof}
1. $\Rightarrow$ 2. is a case of Proposition \ref{hausdorfftofunctionallyhausdorff}. 2.  $\Rightarrow$ 3. follows from the fact that $u_{\ast}:\piztop\ra \pitop$ is a continuous injection. 3. $\Rightarrow$ 1. follows from Theorem \ref{theoremfreetopologicalgroups} and 1. of Lemma \ref{freetopgrpfacts}.
\end{proof}
Note that when $dim(X)=0$, or more generally when $X$ is totally path disconnected, the isomorphism of Theorem \ref{theoremfreetopologicalgroups} simplifies to $\pitsus\cong \fmx$. In this case, $\sus$ provides a particularly nice geometric interpretation of $\fmx$.
\begin{remark} \emph{
In the case $X=\Omega(Y)$, the counit $\Sigma(\Omega(Y)_+)\ra Y$ induces the multiplication map $m_{\pitopy}:F_{M}(\pitopy)\cong \pi_{1}^{\tau}\left(\Sigma(\Omega(Y)_+)\right)\ra \pi_{1}^{\tau}(Y)$ used to define the topology of $\pi_{1}^{\tau}(Y)$.}
\end{remark}
\begin{example} \label{exampleconvergingcircles} \emph{
Let $\omega$ be the discrete space of natural numbers. We find an interesting case when $X=\{1,2,...,\infty\}$ is the one-point compactification of the natural numbers. We write $X=\omega+1$ when we wish to view $X$ as the first compact, infinite ordinal. The generalized wedge $\sus$ is homeomorphic to the planar continuum $$S^1\cup\bigcup_{n\geq 1}\left\{(x,y)\in \mathbb{R}^{2}|\left(x-\frac{1}{n}\right)^{2}+y^{2}=\left(1+\frac{1}{n}\right)^{2}\right\}$$
\begin{figure}[H] \centering \includegraphics[height=2.7cm]{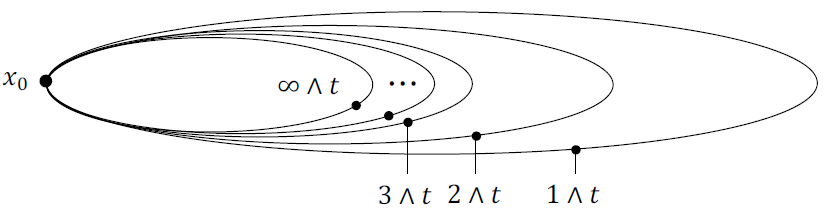} 
\caption{The generalized wedge on the one-point compactification of the natural numbers.}
\end{figure}
\indent Let $dX$ be the underlying set of $X$ with the discrete topology. The identity $dX\ra X$ induces a weak equivalence $\bigvee_{dX}S^1\cong \Sigma((dX)_{+})\ra \sus$, but $\pi_{1}^{\tau}\left( \bigvee_{dX}S^1\right)$ is discrete by Corollary \ref{discretenessgeneral} and $\pitsus\cong F_{M}(\omega+1)$ is not discrete. Therefore $\sus$ is not homotopy equivalent to a countable wedge of circles.\\
\indent One can generalize this ability to distinguish homotopy types of one-dimensional spaces with isomorphic fundamental groups by noting that if two zero-dimensional spaces $X$ and $Y$ have the same cardinality but are such that $F_{M}(X)\ncong F_{M}(Y)$, then $\pi_{1}(\sus)\cong F(X)\cong F(Y)\cong \pi_{1}(\Sigma(Y_+))$ but $\pitsus\ncong \pi_{1}^{\tau}(\Sigma(Y_+))$.
}
\end{example}
Let $X$ and $Y$ be countably infinite compact Hausdorff spaces. Any such space is homeomorphic to a countable successor ordinal (with the order topology) by the Mazurkiewicz-Sierpinski Theorem \cite{MazSier20} and embeds into $\rat$. Thus $\sus$ and $\Sigma(Y_+)$ are one-dimensional (but non-Peano) planar continua with $\pi_{1}(\sus)\cong F(X)\cong F(Y)\cong \pi_{1}(\Sigma(Y_+))$. Baars \cite{Baars92} gives the following characterization using a result of Graev \cite{Gr62}: $\fmx$ and $\fmy$ are isomorphic topological groups if and only if there are countable ordinals $\alpha,\beta$ such that $X\cong \alpha+1$, $Y\cong \beta+1$, and $\max(\alpha,\beta)<\left(\min(\alpha,\beta)\right)^{\omega}$. This result is put to use in the next example.
\begin{example}\emph{
If $X=\omega+1$ and $Y=\omega^{\omega}+1$, then $\sus$ and $\Sigma(Y_+)$ are one-dimensional (but non-Peano) planar continua. Though these spaces have isomorphic fundamental groups (free on countable generators), they cannot be homotopy equivalent. Indeed, if $\alpha$ and $\beta$ are ordinals homeomorphic to $\omega$ and $\omega^{\omega}$ respectively, then we necessarily have $\alpha=\omega$ and $\beta=\omega^{\omega}$ \cite{KL06}. But $\max(\alpha,\beta)=\omega^{\omega}=\left(\min(\alpha,\beta)\right)^{\omega}$ and therefore $\pitsus$ and $\pi_{1}^{\tau}(\Sigma(Y_+))$ are non-isomorphic, non-discrete, Hausdorff topological groups.
}
\end{example}
\subsection{Every topological group is a fundamental group}
The fact that every group is realized as a fundamental group is easily arrived at by the process of attaching 2-cells to wedges of circles. Similarly, we attach 2-cells to generalized wedges of circles (Section \ref{generalizedwedges}) to realize every topological group as the fundamental group of some space. First, we note that attaching n-cells to a space changes the topology of the fundamental group in a convenient way. The following Lemma first appeared in \cite{Bi02}; an alternative proof appears in \cite{Br10.1}.
\begin{lemma}  \label{attachinglemma1}
Suppose $Z$ is a based space, $n\geq 2$ an integer, and $f:S^{n-1}\ra Z$ is a based map. Let $Z'=Z\sqcup_{f}e^{n}$ be the space obtained by attaching an n-cell to $Z$ via the attaching map $f$. The inclusion $j:Z\hookrightarrow Z'$ induces a homomorphism $j_{\ast}:\pi_{1}^{qtop}(Z)\ra \pi_{1}^{qtop}(Z')$ which is also a topological quotient map.
\end{lemma}
The compactness of $S^1$ allows us to easily extend Lemma \ref{attachinglemma1} to an arbitrary number of cells.
\begin{lemma} \label{attachinglemma2}
Suppose $Z$ is a based space, $n\geq 2$ an integer, and $f_{\alpha}:S^{n-1}\ra Z$, $\alpha\in A$ is a family of based maps. Let $Z'=Z\sqcup_{f_{\alpha}}e_{\alpha}^{n}$ be the space obtained by attaching n-cells to $Z$ via the attaching maps $f_{\alpha}$. The inclusion $j:Z\hookrightarrow Z'$ induces a homomorphism $j_{\ast}:\pi_{1}^{qtop}(Z)\ra \pi_{1}^{qtop}(Z')$ which is also a topological quotient map.
\end{lemma}
\begin{proof} We re-label $Z=Z_1$ and $Z'=Z_4$ and take the approach of factoring the inclusion $j:Z_1\hookrightarrow Z_4$ twice as $Z_1\subseteq Z_2\subseteq Z_3\subseteq Z_4$. In general, $\pi_k:\Omega(Z_k)\ra \pi_{1}^{qtop}(Z_k)$ denotes the quotient map identifying homotopy classes of maps.\\

Consider the commutative diagram $$\xymatrix{ \Omega(Z_1) \ar[r]^{j_{\#}} \ar[d]_{\pi_1} & \Omega(Z_4) \ar[d]^{\pi_4} \\
\pi_{1}^{qtop}(Z_1) \ar[r]_{j_{\ast}} & \pi_{1}^{qtop}(Z_4) }$$and suppose $U\subseteq \pi_{1}^{qtop}(Z_4)$ such that $j_{\ast}^{-1}(U)$ is open in $\pi_{1}^{qtop}(Z_1)$. It suffices to show that $\pi_{4}^{-1}(U)$ is open in $\Omega(Z_4)$. Let $\beta\in \pi_{4}^{-1}(U)$. Since the image $\beta(S^1)$ is compact, it intersects only finitely many of the attached cells. Suppose $\alpha_1,...,\alpha_N$ are the indices in $A$ such that $\beta(S^1)\cap e^{n}_{\alpha_{i}}\neq \emptyset$. Let $Z_2=Z_1\sqcup_{\alpha_i} e^{n}_{\alpha_i}\subseteq Z_4$ be the subspace of $Z_4$ which is $Z_1$ with the cells $e^{n}_{\alpha_1},...e^{n}_{\alpha_N}$ attached. Additionally, for each $\alpha\in A-\{\alpha_1,...\alpha_N\}$, take a point $z_{\alpha}\in int(e^{n}_{\alpha})$ and let $Z_3=Z_4-\{z_{\alpha}|\alpha\in A-\{\alpha_1,...\alpha_N\}\}$ be the open subspace of $Z_4$ with the chosen interior points removed. We know from Lemma \ref{attachinglemma1} that the inclusion $j_1:Z_1\hookrightarrow Z_2$ induces a quotient map $(j_1)_{\ast}:\pi_{1}^{qtop}(Z_1)\ra \pi_{1}^{qtop}(Z_2)$ since $Z_2$ is obtained from $Z_1$ by attaching only finitely many n-cells. The inclusion $j_2:Z_2\hookrightarrow Z_3$ is a homotopy equivalence and therefore induces an isomorphism $(j_2)_{\ast}:\pi_{1}^{qtop}(Z_2)\ra \pi_{1}^{qtop}(Z_3)$ of quasitopological groups. Lastly, since $S^1$ is compact and $Z_3$ is open in $Z_4$, the inclusion $j_3:Z_3\hookrightarrow Z_4$ induces an open embedding $(j_3)_{\#}:\Omega(Z_3)\hookrightarrow \Omega(Z_4)$ on loop spaces. Overall, we have $j_3\circ j_2\circ j_1= j$ and that $(j_2\circ j_1)_{\ast}=(j_2)_{\ast}\circ(j_1)_{\ast}:\pi_{1}^{qtop}(Z_1)\ra \pi_{1}^{qtop}(Z_3)$ is a quotient map. The equality $$j_{\ast}^{-1}(U)=(j_2\circ j_1)_{\ast}^{-1}((j_3)_{\ast}^{-1}(U))$$then implies that $(j_3)_{\ast}^{-1}(U)$ is open in $\pi_{1}^{qtop}(Z_3)$. Therefore, $V=\pi_{3}^{-1}((j_3)_{\ast}^{-1}(U))=(j_3)_{\#}^{-1}(\pi_{4}^{-1}(U))$ is an open neighborhood of $\beta$ in $\Omega(Z_3)$. Since $(j_3)_{\#}:\Omega(Z_3)\hookrightarrow \Omega(Z_4)$ is an open embedding, $(j_3)_{\#}(V)$ is an open neighborhood of $\beta$ in $\Omega(Z_4)$. If $\gamma\in (j_3)_{\#}(V)$, then we have a loop $\gamma '\in V$ such that $j_3\circ \gamma '=\gamma$. But this means $[\gamma ']\in (j_3)_{\ast}^{-1}(U)$, so that $[\gamma ]=[j_3\circ \gamma ']\in U$ and consequently $\gamma \in \pi_{4}^{-1}(U)$. This gives the inclusion $(j_3)_{\#}(V)\subseteq \pi_{4}^{-1}(U)$ and that $\pi_{4}^{-1}(U)$ is open in $\Omega(Z_4)$.
\end{proof}
Since $\tau$ preserves quotient maps (Proposition \ref{discoftaugrp}), we have:
\begin{corollary}  \label{attachinglemma3}
Suppose $Z$ is a based space, $n\geq 2$ an integer, and $f_{\alpha}:S^{n-1}\ra Z$, $\alpha\in A$ is a family of based maps. Let $Z'=Z\sqcup_{f_{\alpha}}e_{\alpha}^{n}$ be the space obtained by attaching n-cells to $Z$ via the attaching maps $f_{\alpha}$. The inclusion $j:Z\hookrightarrow Z'$ induces a quotient map $j_{\ast}:\pi_{1}^{\tau}(Z)\ra \pi_{1}^{\tau}(Z')$ of topological groups.
\end{corollary}
We use this Corollary and the realization of free topological groups as fundamental groups (Theorem \ref{theoremfreetopologicalgroups}) to construct, for any topological group $G$, a space $Z$ whose fundamental group is isomorphic to $G$.
\begin{theorem} \label{realizingtopgrps}
Every topological group $G$ is isomorphic to the fundamental group $\pi_{1}^{\tau}(Z)$ of a space $Z$ obtained by attaching 2-cells to a generalized wedge of circles $\sus$. Moreover, one may continue to attach cells of dimension $> 2$ to obtain a space $Z'$ such that $\pi_{1}^{\tau}(Z')\cong G$ and $\pi_{n}^{\tau}(Z')=0$ for $n> 1$.
\end{theorem}
\begin{proof} Suppose $G$ is a topological group. According to the main result of \cite{Har80}, there is a (paracompact Hausdorff) space $X$ such that $\pi_{0}^{qtop}(X)$ is homeomorphic to the underlying space of $G$. The construction of $X$ is functorial and the homeomorphism $\pi_{0}^{qtop}(X)\cong G$ is natural, however, $X$ does not seem to inherit any natural algebraic structure. In the case that $G$ is totally path disconnected, take $X=G$. This gives natural isomorphisms \[h_G:\pitsus\cong F_{M}(\piztop)\cong F_{M}(G).\]
\indent The identity $G\ra G$ induces the quotient map $m_G:F_{M}(G)\ra G$ so that $m_G\circ h_G:\pitsus\cong F_{M}(G)\ra G$ is a quotient map of topological groups. For each $g\in G\cong \pi_{0}^{qtop}(X)$ fix a point $x_g\in X$ such that the path component of $x_g$ corresponds to $g$. Recall that $u_{x_g}:I\ra \sus$ is the loop $u_{x_g}(t)=x_g\wedge t$ and the set of homotopy classes $\{[u_{x_g}]|g\in G\}$ freely generate $\pi_{1}(\sus)\cong F(G)$. For each $\alpha\in \ker (m_G\circ h_G)$ choose a representative loop 
$f_{\alpha}=u_{x_{g_1}}\ast u_{x_{g_2}}\ast\dots \ast u_{x_{g_n}}:I\ra \sus$ and attach a 2-cell to $\sus$ via $f_{\alpha}$. If $Z=\sus\sqcup_{f_{\alpha}} e^{2}_{\alpha}$ is the resulting space, Corollary \ref{attachinglemma3} gives that the inclusion $j:\sus\hookrightarrow Z$ induces a quotient map $j_{\ast}:\pitsus \ra \pi_{1}^{\tau}(Z)$. Since $\ker (m_G\circ h_G) =\ker j_{\ast}$ and both $m_G\circ h_G$ and $j_{\ast}$ are quotient, $\pi_{1}^{\tau}(Z)\cong G$ as topological groups.\\
\indent The second statement of the Theorem follows by the usual process of inductively killing the n-th homotopy group by attaching cells of dimension $n+1$. The fact that the inclusion at each step induce group isomorphisms on the fundamental group (which are topological quotients by Lemma \ref{attachinglemma2} and therefore homeomorphisms) means the direct limit space $Z'$ will satisfy $\pi_{1}^{\tau}(Z')\cong G$ and $\pi_{n}^{\tau}(Z')=0$ for all $n> 1$.
\end{proof}
In the construction of $Z'$ in the previous theorem one will notice that $Z'$ is a CW-complex (and therefore a proper $K(G,1)$) if and only if $X=G$ is a discrete group. This theorem also permits the odd phenomenon of taking non-trivial fundamental groups of fundamental groups. For instance, there is a space $X$ such that $\pitx\cong S^1$ and thus $\pi_{1}^{\tau}(\pitx)\cong \mathbb{Z}$.
\subsection{A topological van Kampen theorem}
In this section, we prove a computational result analogous to the classical Seifert-van Kampen Theorem for fundamental groups. The many variations of the van Kampen theorem are also likely to have topological analogues. Consider the following general statement.
\begin{statement} \label{generalvankamp}
Let $(X,x_0)$ be a based space and $\{U_1,U_2,U_1\cap U_2\}$ be an open cover of $X$ consisting of path connected neighborhoods each containing $x_0$. The diagram $$\xymatrix{ \pi_{1}^{\tau}(U_1\cap U_2) \ar[r] \ar[d] & \pi_{1}^{\tau}(U_1) \ar[d] \\ \pi_{1}^{\tau}(U_2) \ar[r] & \pitx }$$induced by inclusions is a pushout in the category of topological groups. 
\end{statement}
Unfortunately, this statement does not hold in full generality.
\begin{example} \label{vankampex} \emph{
Let $X=\{1,2,...\}\cup \{\infty\}$ be the one-point compactification of the discrete space of natural numbers and $\sus$ be the generalized wedge of Example \ref{exampleconvergingcircles}. A basic open neighborhood of a point $x\wedge t\in \sus-\{x_0\}$ is $U\wedge (a,b)=\{u\wedge s|u\in U,s\in S\}$ where $U$ is an open neighborhood of $x$ in $X$ and $t\in (a,b)\subseteq (0,1)$. The contractible subspaces $X\wedge ([0,\epsilon)\cup (1-\epsilon,1])$, $\epsilon<\frac{1}{2}$ form a neighborhood base at the canonical basepoint $x_0$. We construct a space $Y$ by attaching 1-cells to $\sus$. For each $x\in X$, let $f_x:S^0\ra \sus$ be the map given by $f_{x}(-1)=x_0$ and $f_{x}(1)=x\wedge \frac{1}{2}$. Let $Y=\sus\sqcup_{f_x}e^{1}_{x}$ be the space obtained by attaching a copy of the interval $e_{x}^{1}=[-1,1]$ for each $x$ via the attaching map $f_x$.
\begin{figure}[H] \centering \includegraphics[height=3.8cm]{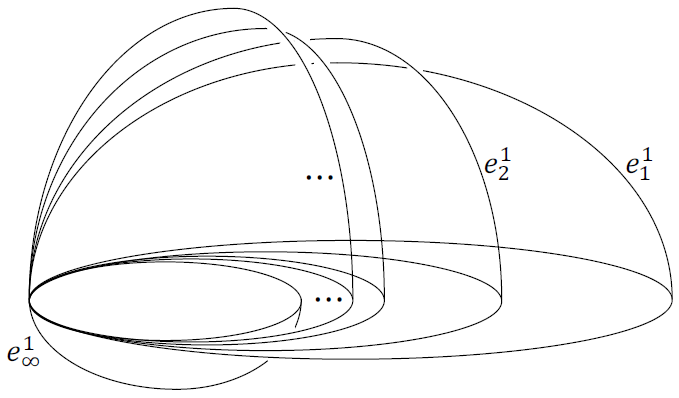} 
\caption{$Y=\sus\sqcup_{f_x}e^{1}_{x}$}
\end{figure}
Note that any open neighborhood of the loop $\alpha:I\ra \sus\subset Y$, $\alpha(t)=\infty\wedge t$ contains loops which are not homotopic to $\alpha$. Thus $\pity$ is not discrete by Proposition \ref{discretenessgeneral}. Define an open cover of $Y$ by letting $$U_1=\left(X\wedge \left(\left[0,\frac{1}{6}\right)\cup \left(\frac{2}{6},1\right]\right)\right)\cup \bigcup_{x\in X}e^{1}_{x}\text{  and  }U_2=\left(X\wedge \left(\left(\frac{5}{6},1\right]\cup \left[0,\frac{4}{6}\right)\right)\right)\cup \bigcup_{x\in X}e^{1}_{x}$$Note that $U_1\cong U_2$.
\begin{figure}[H] \label{openone} \centering \includegraphics[height=4cm]{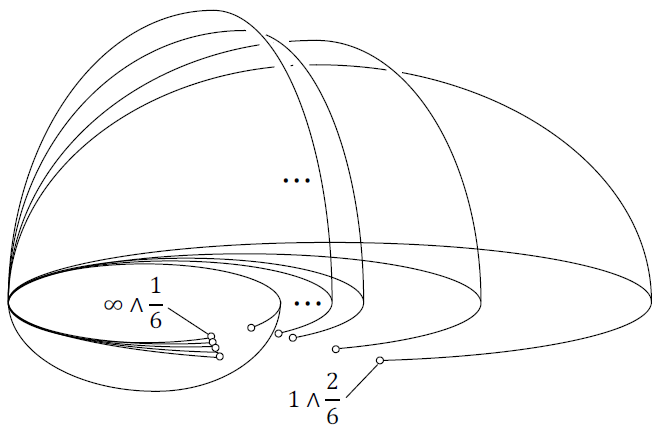} \caption{The open set $U_1\subset Y$}\end{figure}
Collapsing the set $\sus\cap U_1$ to a point gives map $U_1\ra \bigvee_{X}S^1$ to a countable wedge of circles which induces an isomorphism $\pi_{1}^{\tau}(U_1)\ra \pi_{1}^{\tau}\left(\bigvee_{X}S^1\right)$ of discrete topological groups. Consequently, $\pi_{1}^{\tau}(U_1)\cong \pi_{1}^{\tau}(U_2)$ is the discrete free group on countably many generators. Note that $$U_1\cap U_2=\left(X\wedge \left[0,\frac{1}{6}\right)\cup \left(\frac{2}{6},\frac{4}{6}\right)\cup \left(\frac{5}{6},1\right]\right)\cup \bigcup_{x\in X}e^{1}_{x}$$
\begin{figure}[H] \centering \includegraphics[height=3.8cm]{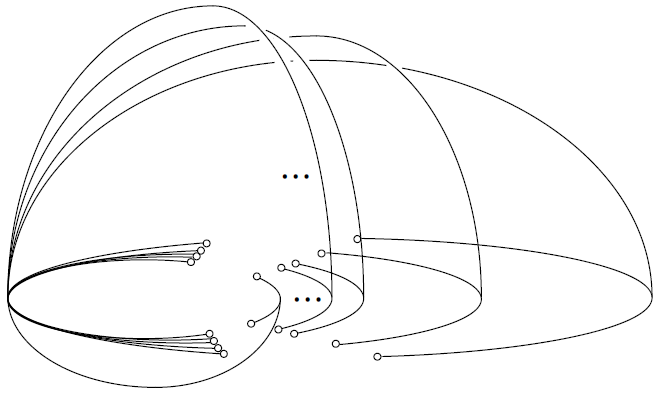} \caption{$U_1\cap U_2$}\end{figure}
Clearly $\pi_{1}(U_1\cap U_2)=1$. If the square$$\xymatrix{ \pi_{1}^{\tau}(U_1\cap U_2) \ar[r] \ar[d] & \pi_{1}^{\tau}(U_1) \ar[d] \\ \pi_{1}^{\tau}(U_2) \ar[r] & \pity }$$is a pushout in the category of topological groups, then $\pity$ is the free topological product of two discrete groups and must also be discrete. This contradiction indicates that Statement \ref{generalvankamp} cannot be true in full generality.}
\end{example}
The complication arising in the previous example motivates the following definition.
\begin{definition} \label{lwe} \emph{A path $p:I\ra X$ is \textit{well-ended} if for every open neighborhood $\mathcal{U}$ of $p$ in $\px$ there are open neighborhoods $V_0,V_1$ of $p(0),p(1)$ in $X$ respectively such that for every $a\in V_0, b\in V_1$ there is a path $q\in \mathcal{U}$ with $q(0)=a$ and $q(1)=b$. A space $X$ is \textit{wep-connected} if for every $x,y\in X$, there is a well-ended path from $x$ to $y$.}
\end{definition}
\begin{remark}
\emph{Since a basis for the topology of $\px$ is given by neighborhoods of the form $\bigcap_{j=1}^{n}\langle K_{n}^{j},U_j\rangle$ where $U_j$ is open in $X$, we have: A path $\alpha\in \px$ is well-ended if and only if for each neighborhood $\bigcap_{j=1}^{n}\langle K_{n}^{j},U_j\rangle$ of $\alpha$, there are open neighborhoods $V_0\subseteq U_1,V_1\subseteq U_n$ of $\alpha(0),\alpha(1)$ in $X$ respectively such that for every $a\in V_0, b\in V_1$ there is a path $\beta\in \bigcap_{j=1}^{n}\langle K_{n}^{j},U_j\rangle$ with $\beta(0)=a$ and $\beta(1)=b$. It is instructive to observe this in the following figure.}
\end{remark}
\begin{figure}[H] \label{fig1}  \centering \includegraphics[height=3cm]{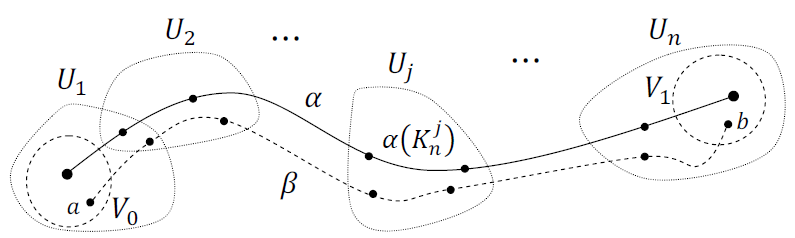} \caption{A well-ended path $\alpha$.} \end{figure}
Well ended-paths play an important role in the theory of semicoverings \cite{Brsemi}. We avoid an involved discussion of wep-connected spaces by restricting to what is necessary for Theorem \ref{vankampentheorem}. We first note that wep-connectedness is a natural generalization of local path connectedness.
\begin{proposition}  \label{lpcimplieslwepc}
If $p:I\ra X$ is a path and $X$ is locally path connected at $p(0)$ and $p(1)$, then $p$ is well-ended. Consequently, all path connected, locally path connected spaces are wep-connected.
\end{proposition}
\begin{proof}
Suppose $\mathcal{U}=\bigcap_{j=1}^{n}\langle K_{n}^{j},U_j\rangle$ is a basic open neighborhood of $p$ in $\px$. Find a path connected neighborhood $V_0,V_1$ of $p(0),p(1)$ respectively such that $V_0\subseteq U_1$ and $V_1\subseteq U_n$. For points $a\in V_0,b\in V_1$, take paths $\alpha:I\ra V_0$ from $a$ to $p(0)$ and $\beta:I\ra V_1$ from $p(1)$ to $b$. Now define a path $q\in \mathcal{U}$ by setting $$q_{K_{2n}^{1}}=\alpha, q_{K_{2n}^{2}}=p_{K_{n}^{1}}, q_{\left[\frac{1}{n},\frac{n-1}{n}\right]}=p_{\left[\frac{1}{n},\frac{n-1}{n}\right]}\text{ ,  }q_{K_{2n}^{2n-1}}=p_{K_{n}^{n-1}}\text{ ,  and  }q_{K_{2n}^{2n}}=\beta$$Clearly $q$ is a path in $\mathcal{U}$ from $a$ to $b$.
\end{proof}
A straightforward partial converse to Proposition \ref{lpcimplieslwepc} is that a space $X$ is locally path connected if and only if every constant path $I\ra X$ is well-ended. Many non-locally path connected spaces are wep-connected. 
\begin{example}
\emph{
It is an easy exercise to verify that for every space $X$, the generalized wedge $\sus$ is wep-connected; the canonical choice of well-ended paths $I\ra \sus$ are those which have image in a circle $x\wedge I\cong S^1$ for some $x\in X$. For instance, the generalized wedge in Example \ref{exampleconvergingcircles} is wep-connected but is not locally path connected. More generally, spaces, such as $Z$ and $Z'$ in Theorem \ref{realizingtopgrps}, obtained by attaching cells to a generalized wedge are wep-connected but are not typically locally path connected.
}
\end{example}
It is worthwhile to observe two examples of space which are not wep-connected.
\begin{example} \label{zeemansexample} \emph{
Zeeman's example \cite[6.6.14]{HW60} illustrated below is not wep-connected since there are no well-ended paths starting or ending at the two points where the space fails to be locally path connected.}
\end{example}
\begin{figure}[H]  \label{fig6}  \centering \includegraphics[height=3cm]{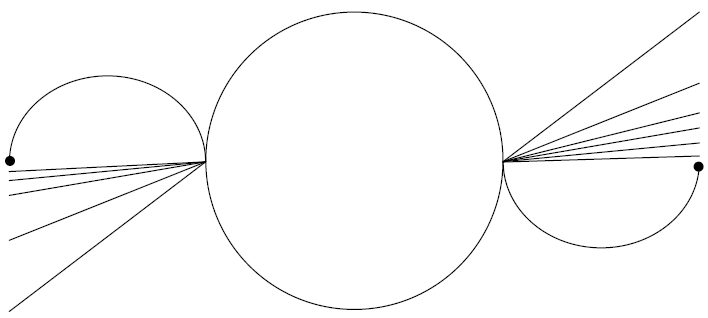} \caption{Zeeman's example} \end{figure} 
\begin{example} \emph{
In Example \ref{vankampex}, the space $Y$ is wep-connected, however, the intersection $U_1\cap U_2$ is not. It turns out that this is precisely why Statement \ref{generalvankamp} fails to hold for the given choice of $U_1$ and $U_2$. Additionally, this situation illustrates that a path connected, open subspace of a wep-connected space need not be wep-connected.
}
\end{example}
\begin{remark} \label{lwecodomain} \emph{
In Definition \ref{lwe}, it is necessary to specify the codomain $X$ since it may occur that a path $p:I\ra A$ in a subspace $A\subseteq X$ is well-ended whereas $p:I\ra A\hookrightarrow X$ is not. This complication does not arise when $A$ is an open subset of $X$ since, whenever $A$ is open, $P(A)$ is an open subspace of $\px$.
}
\end{remark}
\begin{proposition} \label{lweinopensubspace}
If $A$ is open in $X$ and $p:I\ra A$ is a path, then $p:I\ra A$ is well-ended if and only if $p:I\ra A\hookrightarrow X$ is well-ended.
\end{proposition}
The following lemma is useful since the proof of the van Kampen theorem requires a fixed basepoint.
\begin{lemma} \label{lemmalwelwt} Fix any $x_0\in X$. Then $X$ is wep-connected if and only if for each $x\in X$ there is a path $p:I\ra X$ from $x_0$ to $x$ such that for every open neighborhood $\mathcal{U}$ of $p$ in $P(X,x_0)$, there is an open neighborhood $V$ of $x$ such that for each $v\in V$, there is a path $q\in \mathcal{U}$ from $x_0$ to $v$.
\end{lemma}
\begin{proof}
Clearly, if $X$ is wep-connected, the second statement holds. For the converse, pick $a,b\in X$. By assumption there are paths $\alpha,\beta\in P(X,x_0)$ ending at $a,b$ respectively each satisfying the conditions in the statement of the lemma. We claim $\alpha\ast\beta^{-1}$ is well-ended. Let $\mathcal{U}=\bigcap_{j=1}^{n}\langle K_{n}^{j},U_j\rangle$ be a basic open neighborhood of $\alpha\ast\beta^{-1}$ in $\px$. Since $\mathcal{A}=\mathcal{U}_{\left[0,\frac{1}{2}\right]}\cap P(X,x_0)$ and $\mathcal{B}=\left(\mathcal{U}_{\left[\frac{1}{2},1\right]}\right)^{-1}\cap P(X,x_0)$ are open neighborhoods of $\alpha$ and $\beta$ respectively, there are open neighborhoods $A$ of $a$ and $B$ of $b$ such that for any $a'\in A$ (resp. $b'\in B$) there is a path $\alpha '\in \mathcal{A}$ (resp. $\beta '\in \mathcal{B}$) from $x_0$ to $a'$ (resp. $x_0$ to $b'$). Now $\alpha '\ast (\beta ')^{-1}\in \mathcal{U}$ is the desired path from $a'$ to $b'$.
\end{proof}
It is thus convenient to give the following definition.
\begin{definition}
\emph{
A path $p:I\ra X$ is \textit{well-targeted} if for every open neighborhood $\mathcal{U}$ of $p$ in $P(X,p(0))$, there is an open neighborhood $V$ of $p(1)$ such that for each $v\in V$, there is a path $q\in \mathcal{U}$ from $p(0)$ to $v$.
}
\end{definition}
Statement \ref{generalvankamp} is now proven in the case that the intersection $U_1\cap U_2$ is wep-connected.
\begin{vankampenthm} \label{vankampentheorem} 
Let $(X,x_0)$ be a based space and $\{U_1,U_2,U_1\cap U_2\}$ an open cover of $X$ consisting of path connected neighborhoods each containing $x_0$. Let $k_i:U_1\cap U_2\hookrightarrow U_i$ and $l_i:U_i\hookrightarrow X$ be the inclusions. If $U_1\cap U_2$ is wep-connected, the induced diagram of continuous homomorphisms $$\xymatrix{ \pi_{1}^{\tau}(U_1\cap U_2) \ar[r]^-{(k_1)_{\ast}} \ar[d]_-{(k_2)_{\ast}} & \pi_{1}^{\tau}(U_1) \ar[d]^-{(l_1)_{\ast}} \\ \pi_{1}^{\tau}(U_2) \ar[r]_-{(l_2)_{\ast}} & \pitx }$$is a pushout square in the category of topological groups. In other words, there is a canonical isomorphism $$\pitx\cong \pi_{1}^{\tau}(U_1)\ast_{\pi_{1}^{\tau}(U_1\cap U_2)}\pi_{1}^{\tau}(U_2)$$of topological groups.
\end{vankampenthm}
\begin{proof}
We show: If $G$ is a topological group and $f_i:\pi_{1}^{\tau}(U_i)\ra G$ are continuous homomorphisms such that $f_1\circ (k_1)_{\ast}=f_2\circ (k_2)_{\ast}$, there is a unique, continuous homomorphism $\Phi:\pitx\ra G$ such that $\Phi\circ (l_i)_{\ast}=f_i$. The classical van Kampen theorem guarantees the existence and uniqueness of the homomorphism $\Phi$ and so it suffices to show $\Phi$ is continuous. To do this, we show the composition $\phi=\Phi\circ \pi:\ox\ra \pitx\ra G$ is continuous. If this can be done, Proposition \ref{finesttopology} guarantees the continuity of $\Phi:\pitx\ra G$.\\
\indent We first recall a convenient description of $\Phi$: Given any loop $\alpha\in \ox$, find a subdivision $0=t_0<t_1<...<t_n=1$ such that $\alpha_j=\alpha_{[t_{j-1},t_j]}$ has image in $U_{i_j}$ for $i_j\in \{1,2\}$. For $j=1,...,n-1$ find a path $p_{j}:I\to U_j$ from $x_0$ to $a_j=\alpha(t_j)$. If $a_j\in U_1\cap U_2$, we demand that $p_j$ has image in $U_1\cap U_2$. Let $p_0=p_n=c_{x_0}$ be the constant path and $L_j=p_{j-1}\ast \alpha_j\ast p_{j}^{-1}$ for $j=1,...,n$. Note that $L_j$ has image in $U_{i_j}$ and if $\alpha_j$ has image in $U_1\cap U_2$, then so does $L_j$. Additionally, $[\alpha]$ is the product $[L_1]\cdots [L_n]$ in $\pi_1(X)$. Now $\Phi([\alpha])$ is defined to be the product$$f_{i_1}([L_1])f_{i_2}([L_2])\dots f_{i_n}([L_n]).$$ That $\Phi$ is a well-defined homomorphism follows from arguments found in most first course algebraic topology textbooks \cite{Massey91}. It is particularly useful to the current proof that the given description of $\Phi([\alpha])$ does not depend upon the choice of subdivision $0=t_0<t_1<...<t_n=1$ or paths $p_1,...,p_{n-1}$. 

To see that $\phi$ is continuous, suppose $W$ is open in $G$ and $\alpha\in \phi^{-1}(W)$. Write $\phi(\alpha)=\Phi([\alpha]) =f_{i_1}([L_1])f_{i_2}([L_2])\dots f_{i_n}([L_n])$ as above. In particular, choose the subdivision $0=t_0<t_1<...<t_n=1$ so that $i_j\neq i_{j+1}$. This gives $a_j\in U_1\cap U_2$ for each $j$ and thus each $p_j$ has image in $U_1\cap U_2$. Since $p_j$ has image in wep-connected neighborhood $U_1\cap U_2$, we may assume the paths $p_1,...,p_{n-1}$ are well-targeted.

\indent We construct an open neighborhood $\mathcal{U}$ of $\alpha$ contained in $\phi^{-1}(W)$ by combining neighborhoods of its restrictions. Since $G$ is a topological group and $f_{i_1}([L_1])f_{i_2}([L_2])\dots f_{i_n}([L_n])]\in W$, there exists open neighborhoods $W_j$ of $f_{i_j}([L_j])$ in $G$ such that $W_1W_2...W_n\subseteq W$. That both compositions $f_i\circ \pi_{i}:\Omega (U_{i})\ra \pi_{1}^{\tau}(U_i)\ra G$ are continuous means that, for each $j$, there is a basic open neighborhood $\mathcal{V}_j=\bigcap_{m=1}^{M_j}\langle K_{M_j}^{m}, A_{m}^{j}\rangle$ of $L_j$ contained in $\pi_{i_j}^{-1}(f_{i_j}^{-1}(W_j))\subseteq\Omega (U_{i_j})$. We may assume that $M_j$ is divisible by $3$ and that $A_{m}^{j}\subseteq U_1\cap U_2$ whenever $L_j\left(K_{M_j}^{m}\right)\subseteq U_1\cap U_2$. Since each $p_{j}$ has image in $U_1\cap U_2$, this assumption means that $A_{m}^{j}\subseteq U_1\cap U_2$ whenever $K_{M_j}^{m}\subseteq \left[0,\frac{1}{3}\right]\cup \left[\frac{2}{3},1\right]$.\\
\indent Taking restricted neighborhoods, we find that $\left(\mathcal{V}_j\right)_{\left[\frac{2}{3},1\right]}$ is an open neighborhood of $p_{j}^{-1}$ for $j=1,...n-1$, $\left(\mathcal{V}_j\right)_{\left[\frac{1}{3},\frac{2}{3}\right]}$ is an open neighborhood of $\alpha_j$ for $j=1,...,n$, and $\left(\mathcal{V}_j\right)_{\left[0,\frac{1}{3}\right]}$ is an open neighborhood of $p_{j-1}$ for $j=2,...,n$. For $j=1,...,n-1$ both $\left(\left(\mathcal{V}_{j}\right)_{\left[\frac{2}{3},1\right]}\right)^{-1}$ and $\left(\mathcal{V}_{j+1}\right)_{\left[0,\frac{1}{3}\right]}$ are neighborhoods of $p_j$ so we may assume they are equal.

\indent Since $p_{j}:I\ra U_1\cap U_2$ is well-targeted for each $j=1,...,n-1$ with endpoint $p_{j}(1)=a_j$, there is an open neighborhood $B_j$ of $a_j$ in $U_1\cap U_2$ such that for each $b\in B_j$, there is a path $\delta\in \left(\mathcal{V}_{j+1}\right)_{\left[0,\frac{1}{3}\right]}$ from $x_0$ to $b$. Construct the neighborhood $$\mathcal{U}=\bigcap_{j=1}^{n}\left(\left(\mathcal{V}_j\right)_{\left[\frac{1}{3},\frac{2}{3}\right]}\right)^{[t_{j-1},t_j]}\cap \bigcap_{j=1}^{n-1}\langle \{t_{j}\},B_j\rangle$$of $\alpha$ in $\Omega (X,x_0)$. For any loop $\gamma\in \mathcal{U}$, notice that
\begin{itemize}
\item For each $j=1,...,n$, $$\left(\mathcal{V}_j\right)_{\left[\frac{1}{3},\frac{2}{3}\right]}= \left(\left(\left(\mathcal{V}_j\right)_{\left[\frac{1}{3},\frac{2}{3}\right]}\right)^{[t_{j-1},t_j]}\right)_{[t_{j-1},t_j]}$$is an open neighborhood of $\gamma_j=\gamma_{[t_{j-1},t_j]}$ in $\px$.
\item Since $\alpha_j$ has image in $U_{i_j}$, so does $\gamma_{j}$.
\item If $\alpha_j$ has image in $U_1\cap U_2$, then so does $\gamma_{j}$.
\item For $j=1,...,n-1$, since $\gamma(t_{j})\in B_j$, there is a path $$\delta_j\in \left(\mathcal{V}_{j+1}\right)_{\left[0,\frac{1}{3}\right]}=\left(\left(\mathcal{V}_{j}\right)_{\left[\frac{2}{3},1\right]}\right)^{-1}\subseteq \langle I,U_1\cap U_2\rangle$$ from $x_0$ to $\gamma(t_{j})$.
\end{itemize}
Let $\delta_0=\delta_n=c_{x_0}$ and define a loop $\beta$ by demanding that $\beta_{[t_{j-1},t_j]}$ is the loop $\delta_{j-1}\ast \gamma_{[t_{j-1},t_j]} \ast \delta_{j}^{-1}\in \Omega (U_{i_j},x_0)$ for $j=1,...,n$. Note that $\beta_{[t_{j-1},t_j]}$ has image in $U_{i_j}$ and if $\alpha_j$ has image in $U_1\cap U_2$, then so does $\beta_{[t_{j-1},t_j]}$. Thus $$\beta\simeq \left(\delta_{0}\ast\gamma_{[t_{0},t_1]} \ast \delta_{1}^{-1}\right)\ast\dots \ast \left(\delta_{j-1}\ast \gamma_{[t_{j-1},t_j]} \ast \delta_{j}^{-1}\right)\ast\left( \delta_{j}\ast \gamma_{[t_{j},t_{j+1}]} \ast \delta_{j+1}^{-1} \right)\ast\dots \ast\left(\delta_{n-1}\ast \gamma_{[t_{n-1},t_n]}\ast \delta_{n}\right)  \simeq \gamma$$and $\Phi([\beta])$ is well-defined as the product \[f_{i_1}([\beta_{[t_{0},t_1]}])f_{i_2}([\beta_{[t_{1},t_2]}])\dots  f_{i_n}([\beta_{[t_{n-1},t_n]}])\] in $G$. Moreover, for $j=1,...,n$, we have $\left(\beta_{[t_{j-1},t_j]}\right)_{C}\in (\mathcal{V}_j)_{C}$ for $C=\left[0,\frac{1}{3}\right],\left[\frac{1}{3},\frac{2}{3}\right],\left[\frac{2}{3},1\right]$. Therefore $$\beta_{[t_{j-1},t_j]}\in \bigcap_{C}\left((\mathcal{V}_j)_{C}\right)^{C}=\mathcal{V}_j\subseteq \pi_{i_j}^{-1}(f_{i_j}^{-1}(W_j))\subseteq \Omega (U_{i_j},x_0)$$All together, we see that $$\phi(\gamma)=\Phi([\gamma])=\Phi([\beta])=f_{i_1}([\beta_{[t_{0},t_1]}])f_{i_2}([\beta_{[t_{1},t_2]}])\dots  f_{i_n}([\beta_{[t_{n-1},t_n]}])\in W_1W_2\dots W_n\subseteq W.$$This completes the proof of the inclusion $\mathcal{U}\subseteq \phi^{-1}(W)$. Therefore $\phi$ is continuous.
\end{proof}
\begin{remark} \emph{
While the condition that $U_1\cap U_2$ be wep-connected is sufficient for the van Kampen theorem to hold, it is certainly not a necessary condition. For any path connected, non-wep-connected space $X$, the unreduced suspension $SX$ is quotient of $X\times I$ by collapsing both $X\times \{0\}$ and $X\times \{1\}$ to points. Let $U_1$ and $U_2$ be the image of $X\times \left[0,\frac{2}{3}\right)$ and $X\times \left(\frac{1}{3},1\right]$ in the quotient respectively. The open sets $U_1,U_2$ are contractible and the van Kampen theorem holds trivially even though $U_1\cap U_2$ is not wep-connected.}
\end{remark}
\begin{example}  \emph{
We can now compute $\pity$ from Example \ref{vankampex} by choosing an appropriate cover. For each $x\in X=\omega+1$, let $0_x$ denote $0$ in $e^{1}_{x}=[-1,1]$. Since $U_3=Y-\bigcup_{x\in X}\{0_x\}$ is homotopy equivalent to $\sus$, we have $\pi_{1}^{\tau}(U_3)\cong F_{M}(\omega+1)$. The union $U_1\cup U_3$ is $X$ and $U_1\cap U_3$ is wep-connected and 1-connected. Therefore, the van Kampen theorem applies and gives $$\pity\cong  \pi_{1}^{\tau}(U_1)\ast\pi_{1}^{\tau}(U_3)\cong  F_{M}(\omega)\ast F_{M}(\omega+1)\cong F_{M}( \omega\sqcup (\omega+1))\cong F_{M}(\omega+\omega)$$ where $\omega+\omega$ is the ordinal sum.}
\end{example}
\begin{example}
\emph{
The canonical isomorphism $F_{M}(A_1\sqcup A_2)\cong F_{M}(A_1)\ast F_{M}(A_2)$ of topological groups is quickly recovered from the van Kampen theorem. Choose a space $Y_i$ such that $\pi_{0}^{qtop}(Y_i)\cong A_i$. Let $X=\Sigma((Y_1\sqcup Y_2)_+)\cong \Sigma((Y_1)_+)\vee \Sigma((Y_2)_+)$, $V_i$ be the wep-connected, contractible neighborhood $Y_i\wedge \left[0,\frac{1}{3}\right)\sqcup \left(\frac{2}{3},1\right]\subset  \Sigma((Y_i)_+)$, $U_1=\Sigma((Y_1)_+)\vee V_{2}$, and $U_2=V_1 \vee\Sigma((Y_2)_+)$. Since $U_1\cap U_2=V_1 \vee V_2$ is wep-connected and contractible the van Kampen theorem gives the middle isomorphism in: $$F_{M}(A_1\sqcup A_2)\cong \pi_{1}^{\tau}(\Sigma((Y_1\sqcup Y_2)_+))\cong \pi_{1}^{\tau}(\Sigma((Y_1)_+))\ast \pi_{1}^{\tau}(\Sigma((Y_2)_+))\cong F_{M}(A_1)\ast F_{M}(A_2).$$The first and third isomorphisms come from Theorem \ref{theoremfreetopologicalgroups} and the fact that $\pi_{0}^{qtop}(Y_1 \sqcup Y_2)=A_1\sqcup A_2$. This observation illustrates why the notion of wep-connected intersection is an appropriate generalization of local path connectedness. Indeed, if we only consider the case where $U_1\cap U_2$ is locally path connected, this application of the van Kampen theorem is restricted to discrete groups.
}
\end{example}
\begin{corollary}
Given the hypothesis of Theorem \ref{vankampentheorem}, the homomorphism $F_{M}(\Omega(U_1))\ast F_{M}(\Omega(U_2)) \ra \pi_{1}^{\tau}(X)$ induced by the canonical maps $\Omega(U_i)\ra \pi_{1}^{\tau}(X)$, $i=1,2$ is a topological quotient map.
\end{corollary}
\begin{proof}
Since $F_{M}(\Omega(U_1))\ast F_{M}(\Omega(U_2))\cong F_{M}(\Omega(U_1)\sqcup\Omega(U_2))$ it suffices to show that $Q:F_{M}(\Omega(U_1)\sqcup\Omega(U_2))\ra \pitx$, $Q(\alpha_1...\alpha_n)=[\alpha_1\ast\dots\ast \alpha_{n}]$ is quotient. Let $\pi_i:\Omega(U_i)\ra \pi_{1}^{qtop}(U_i)$ be the quotient map identifying path components. Since $F_M$ preserves quotients, $F_{M}(\pi_1\sqcup \pi_2)$ is quotient. The map $$k:F_{M}(\pi_{1}^{qtop}(U_1)\sqcup \pi_{1}^{qtop}(U_2))\ra \pi_{1}^{\tau}(U_1)\ast \pi_{1}^{\tau}(U_2)$$ of Proposition \ref{coproductsandtau} is also quotient. Additionally, the projection $k':\pi_{1}^{\tau}(U_1)\ast \pi_{1}^{\tau}(U_2)\ra \pi_{1}^{\tau}(U_1)\ast_{\pi_{1}^{\tau}(U_1\cap U_2)} \pi_{1}^{\tau}(U_2)$ is quotient. Let $h:\pi_{1}^{\tau}(U_1)\ast_{\pi_{1}^{\tau}(U_1\cap U_2)} \pi_{1}^{\tau}(U_2)\cong \pitx$ be the isomorphism of Theorem \ref{vankampentheorem}. The composite $Q=h\circ k' \circ k\circ F_{M}(\pi_1\sqcup \pi_2)$ is quotient since it is the composition of quotient maps.
\end{proof}
It is not true that the wedge of two wep-connected spaces is wep connected. For instance, let \[CX=\frac{I\times X}{\{0\}\times X}\] be the cone on $X=\{1,2,...,\infty\}$ (the one-point compactification of the natural numbers) and have basepoint the image of $(1,\infty)$ in the quotient. Even though $CX$ is wep-connected, $CX\vee CX$ is not wep-connected. Thus to apply the van Kampen theorem to a wedge of two spaces, we require the following lemma.
\begin{lemma} \label{wedgeoflwtspaces}
If $(X_{\lambda},x_{\lambda})$ is a family of based, wep-connected spaces which are locally path connected at their basepoints and such that $\{x_{\lambda}\}$ is closed in $X_{\lambda}$ for each $\lambda$, then $X=\bigvee_{\lambda}X_{\lambda}$ is wep-connected.
\end{lemma}
\begin{proof}
Let $x_0$ be the canonical basepoint of $X$. Since $X$ is locally path connected at $x_0$, the constant path $c_{x_0}:I\ra X$ is well-targeted. Let $z\in X_{\lambda}-\{x_{\lambda}\}$. It suffices to find a well-targeted path in $X$ from $x_0$ to $z$. Since $X_{\lambda}$ is wep-connected, there is a well-targeted path $\gamma:I\ra X_{\lambda}$ from $x_{\lambda}$ to $z$. Let $\mathcal{U}=\bigcap_{j=1}^{n}\langle K_{n}^{j},U_j\rangle$ be an open neighborhood of $\gamma:I\ra X_{\lambda}\hookrightarrow X$ in $P(X,x_0)$. Then $\mathcal{V}=\bigcap_{j=1}^{n}\langle K_{n}^{j},U_j\cap X_{\lambda}\rangle$ is an open neighborhood of $\gamma$ in $P(X_{\lambda},x_0)$. 

Since $\gamma$ is well-targeted, there is an open neighborhood $V$ of $z$ in $X_{\lambda}$ such that $x_{\lambda}\notin V$ and such that for each $v\in V$, there is a $\delta\in \mathcal{V}$ from $x_{\lambda}$ to $z$. Note that $V$ is an open neighborhood of $z$ in $X$ and for each $v\in V$ there is a path $\delta:I\ra X_{\lambda}\hookrightarrow X$ in $\mathcal{U}$ from $x_0$ to $v$. Thus $\delta:I\ra X_{\lambda}\hookrightarrow X$ is well-targeted.
\end{proof}
\begin{theorem}
Let $(X,x_0),(Y,y_0)$ be path connected spaces having a countable neighborhood base of 1-connected neighborhoods at their basepoints. If there are wep-connected, simply connected neighborhoods $A$ of $x_0$ in $X$ and $B$ of $y_0$ in $Y$, then there is a canonical isomorphism $\pi_{1}^{\tau}(X\vee Y)\cong \pitx\ast\pity$ of topological groups.
\end{theorem}
\begin{proof}
We first recall a theorem of Griffiths \cite{Gr54}: If $W_1,W_2$ are based spaces, one of which has a countable base of 1-connected neighborhoods at its basepoint, then the inclusions $W_i\hookrightarrow W_1\vee W_2$ induce an isomorphism $\pi_{1}(W_1)\ast\pi_{1}(W_2)\ra \pi_{1}(W_1\vee W_2)$ of groups. Since $A$, $B$, and $A\vee B$ are open in $X$, $Y$, and $X\vee Y$ respectively, each has a countable neighborhood base of path connected, 1-connected neighborhoods. Griffiths' theorem implies that $\pi_{1}(A \vee B)=1$. Let $U_1=X\vee B$ and $U_2=A\vee Y$ so that $U_1\cap U_2=A\vee B$. Since $A$ and $B$ are wep-connected and locally path connected at their basepoints, $A\vee B$ is wep-connected by Lemma \ref{wedgeoflwtspaces}. The van Kampen theorem applies and gives an isomorphism $\pi_{1}^{\tau}(X\vee Y)\cong \pi_{1}^{\tau}(X\vee B)\ast \pi_{1}^{\tau}(A \vee Y)$ of topological groups. Again using Griffiths' theorem, the inclusions $X\hookrightarrow X\vee B$ and $Y\hookrightarrow A\vee Y$ induce continuous group isomorphisms $\pi_{1}^{\tau}(X)\ra \pi_{1}^{\tau}(X\vee B)$ and $\pi_{1}^{\tau}(Y)\ra \pi_{1}^{\tau}(A\vee Y)$. These group isomorphisms are also homeomorphisms since their inverses are induced by the retractions $X\vee B\ra X$ and $A\vee Y\ra Y$. All together, there are canonical isomorphisms $$\pi_{1}^{\tau}(X\vee Y)\cong \pi_{1}^{\tau}(X\vee B)\ast \pi_{1}^{\tau}(A \vee Y)\cong \pitx\ast\pity.$$
\end{proof}
\section{Conclusions}
Covering space theory, in the classical sense, provides general techniques for studying subgroups of given groups via topology. The three main results in Section 4 indicate that general topological groups, even those with complicated topological structure, are quite naturally realized as fundamental groups of simple wep-connected spaces (which are locally path connected at their basepoints). Thus it is plausible that there are techniques allowing one to study the structure of topological groups by studying the topology of spaces having non-trivial local properties.

The author has proposed a generalization of covering space theory \cite{Brsemi} as one such technique to study open subgroup(oid)s of topologically enriched group(oid)s. This theory (of semicoverings) also indicates that the topology of $\pitx$ typically retains a good deal more information than the category of covering spaces of $X$. It is a great convenience that, unlike covering space theory, semicovering theory applies to arbitrary locally path connected spaces, the generalized wedges of circles of Section 4.1, and the CW-like spaces constructed in Section 4.2. It would be interesting if other properties such as separation and zero dimensionality of topological groups and their connection to shape injectivity could be studied using similar techniques.
\section{Acknowledgements}
The author thanks his advisor Maria Basterra for her guidance and encouragement and Paul Fabel for helpful conversations. Thanks are also due to the University of New Hampshire graduate school for funding in the form of a Dissertation Year Fellowship.

\begin{thebibliography}{10}
\expandafter\ifx\csname url\endcsname\relax
  \def\url#1{\texttt{#1}}\fi
\expandafter\ifx\csname urlprefix\endcsname\relax\def\urlprefix{URL }\fi


\bibitem{AT08}
A.~Arhangel'skii, M.~Tkachenko, Topological Groups and Related Structures,
  Series in Pure and Applied Mathematics, Atlantis Studies in Mathematics,
  2008.

\bibitem{Baars92}
J.~Baars, Equivalence of certain free topological groups, Comment. Math. Univ.
  Carolinae 33 (1992) 125--130.

\bibitem{Bi02}
D.~Biss, The topological fundamental group and generalized covering spaces,
  Topology Appl. 124 (2002) 355--371.


\bibitem{Br10.1}
J.~Brazas, The topological fundamental group and free topological groups,
  Topology Appl. 158 (2011) 779--802.

\bibitem{Br11}
J.~Brazas, University of New Hampshire Ph.D. Dissertation, 2011.
  Dissertation (2011).

\bibitem{Brsemi}
J.~Brazas, Semicoverings: A generalization of covering space theory,
  To appear in Homology, Homotopy Appl.

\bibitem{CM}
J.~Calcut, J.~McCarthy, Discreteness and homogeneity of the topological
  fundamental group, Topology Proc. 34 (2009) 339--349.

\bibitem{CC06}
J.~Cannon, G.~Conner, On the fundamental group of one-dimensional spaces,
  Topology Appl. 153 (2006) 2648--2672.

\bibitem{CS90}
B.~Clark, V.~Schneider, On constructing the associated graev topology, Archiv
  der Math. 55 (1990) 296--297.

\bibitem{CMRZZ08}
G.~Conner, M.~Meilstrup, D.~Repov\u{s}, A.~Zastrow, M.~\u{Z}eljko, On small
  homotopies of loops, Topology Appl. 155 (2008) 1089–-1097.

\bibitem{Eda98}
K.~Eda, K.~Kawamura, The fundamental groups of one-dimensional spaces, Topology
  Appl. 87 (1998) 163--172.

\bibitem{Fab05.3}
P.~Fabel, The topological hawaiian earring group does not embed in the inverse
  limit of free groups, Algebraic and Geometric Topology 5 (2005) 1585--1587.

\bibitem{Fab10}
P.~Fabel, Multiplication is discontinuous in the hawaiian earring group, Bull. Polish Acad. Sci. Math. 59 (2011) 77--83

\bibitem{Fab11}
P.~Fabel, Compactly generated quasitopological homotopy groups with
  discontinuous multiplication, To appear in Topology Proc.

\bibitem{FZ05}
H.~Fischer, A.~Zastrow, The fundamental groups of subsets of closed surfaces
  inject into their first shape groups, Algebraic and Geometric Topology 5
  (2005) 1655–-1676.

\bibitem{FG05}
H.~Fischer, C.~Guilbault, On the fundamental groups of trees of manifolds,
  Pacific J. Math. 221 (2005) 49--79.

\bibitem{FRVZ11}
H.~Fischer, D.~Repovs, Z.~Virk, A.~Zastrow, On semilocally simply connected spaces, Topology Appl. 158 (2011) 397--408.

\bibitem{FZ07}
H.~Fischer, A.~Zastrow, Generalized universal covering spaces and the shape
  group, Fund. Math. 197 (2007) 167--196.

\bibitem{Gr54}
H.~Griffiths, The fundamental group of two spaces with a common point, Quart.
  J. Math. Oxford (2) 5 (1954) 175--190.

\bibitem{GHMM08}
H.~Ghane, Z.~Hamed, B.~Mashayekhy, H.~Mirebrahimi, Topological homotopy groups,
  Bull. Belgian Math. Soc. 15 (2008) 455--464.

\bibitem{GH09}
H.~Ghane, Z.~Hamed, On nondiscreteness of a higher topological homotopy group
  and its cardinality, Bull. Belgian Math. Soc. Simon Stevin 16 (2009)
  179--183.

\bibitem{GHMM10}
F.~Ghane, M.~Hamed, Z., H.~B., Mirebrahimi, On topological homotopy groups of
  n-Hawaiian like spaces, Topology Proc. 36 (2010) 255--266.

\bibitem{Gr62}
M.~Graev, Free topological groups, Amer. Math. Soc. Transl. 8 (1962) 305--365.

\bibitem{Har80}
D.~Harris, Every space is a path component space, Pacific J. Math. 91 (1980)
  95--104.


\bibitem{HW60}
P.~Hilton, S.~Wylie, Homology Theory: An introduction to algebraic topology,
  Cambridge University Press, 1960.

\bibitem{KL06}
V.~Kieftenbeld, B.~Lo\"we, A classification of ordinal topologies, ILLC
  Publications, PP-2006-57.

\bibitem{Mal57}
A.~Mal'tsev, Free topological algebras, Izv. Akad. Nauk SSSR, Ser. Mat. 21
  (1957) 171--198.

\bibitem{MS82}
S.~Mardesic, J.~Segal, Shape Theory, North-Holland Publishing Company, 1982.

\bibitem{Massey91} W.S. Massey, A Basic Course in Algebraic Topology, Springer-Verlag, 1991.

\bibitem{MazSier20}
S.~Mazurkiewicz, W.~Sierpi\'nski, Contribution \'a la topologie des ensembles
  d\'enombrables, Fund. Math. 1 (1920) 17--27.

\bibitem{MM86}
J.~Morgan, I.~Morrison, A van kampen theorem for weak joins, Proc. London Math.
  Soc. 53 (1986) 562--576.

\bibitem{Po91}
H.~Porst, On the existence and structure of free topological groups, Category
  Theory at Work (1991) 165--176.

\bibitem{Sip05}
O.~Sipacheva, The topology of free topological groups, J. Math. Sci. 131 (2005)
  5765–5838.

\bibitem{Th74}
B.~Thomas, Free topological groups, General Topology and its Appl. 4 (1974)
  51--72.


\end{thebibliography}
\end{document}